\documentclass[a4paper,10pt]{amsart}
\usepackage{amsmath,amsthm,amssymb,enumerate}
\usepackage{microtype}
\usepackage{epsfig}
\usepackage[dvipsnames]{xcolor}
\usepackage{esint}
\usepackage{combelow}
\usepackage{bbm}
\usepackage{amssymb}
\usepackage{amsmath}
\usepackage{mathtools}
\usepackage{amssymb}
\usepackage{amsmath,amsthm}
\usepackage{graphicx}
\usepackage{tikz}
\usepackage{tikz-cd}
\usetikzlibrary{calc}
\usepackage{amsfonts}
\usepackage{amsxtra}
\usepackage{euscript,mathrsfs}
\usepackage[left=4cm,right=4cm,top=3.5cm,bottom=2.75cm]{geometry}
\usepackage[unicode, colorlinks=true, linktocpage=true, linkcolor=red!70!black, citecolor=green!50!black]{hyperref}
\let\C\undefined
\allowdisplaybreaks

\usepackage{enumitem}
\setenumerate{label={\rm (\alph{*})}}

\makeatletter
\newcommand*\bigcdot{\mathpalette\bigcdot@{.5}}
\newcommand*\bigcdot@[2]{\mathbin{\vcenter{\hbox{\scalebox{#2}{$\m@th#1\bullet$}}}}}
\makeatother

\newcommand{\R}{\mathbb R}
\newcommand{\N}{\mathbb N}

\newcommand{\Q}{\mathbb Q}

\usepackage{scalerel}
\newcommand\medwedge{\scalerel*{\bigwedge}{H_.}}

\renewcommand{\leq}{\leqslant}

\DeclareMathOperator{\tr}{tr}

\newcommand{\dif}{\mathrm{d}}

\DeclareMathOperator{\diver}{div}

\renewcommand{\dif}{\operatorname{d}\!}
\newcommand{\lebe}{\operatorname{L}}
\newcommand{\sobo}{\operatorname{W}}

\newcommand{\imag}{\operatorname{i}}
\newcommand{\locc}{\operatorname{loc}}

\newcommand{\hold}{\operatorname{C}}

\newcommand{\bv}{\operatorname{BV}}

\newcommand{\ball}{B}
\newcommand{\di}{\operatorname{div}}
\newcommand{\A}{\mathbb{A}}

\newcommand{\bd}{\operatorname{BD}}

\newcommand{\spt}{\operatorname{spt}}
\newcommand{\lin}{\operatorname{Lin}}

\newcommand{\besov}{\operatorname{B}}

\usepackage{esint}

\newcommand{\C}{\mathbb{C}}

\newcommand{\Lin}{\operatorname{Lin}}

\renewcommand{\Re}{\operatorname{Re}}
\renewcommand{\Im}{\operatorname{Im}}

\normalsize

\allowdisplaybreaks

\numberwithin{equation}{section}

\newtheorem{theorem}{Theorem}[section]
\newtheorem{lemma}[theorem]{Lemma}
\newtheorem{proposition}[theorem]{Proposition}

\newtheorem{corollary}[theorem]{Corollary}

\theoremstyle{remark}

\newtheorem{remark}[theorem]{Remark}
\newtheorem{example}[theorem]{Example}

\numberwithin{equation}{section}

\begin{document}
	
	\title[Boundary $\lebe^1$-estimates]{Boundary ellipticity and\\ limiting $\lebe^1$-estimates on halfspaces}
	\author[F. Gmeineder]{Franz Gmeineder}
	\author[B. Rai\cb{t}\u{a}]{Bogdan Rai\cb{t}\u{a}}
	\author[J. Van Schaftingen]{Jean Van Schaftingen}
	\address[F.~Gmeineder]{Universit\"{a}t Konstanz, Fachbereich Mathematik und Statistik, Universit\"{a}tsstra\ss e 10, 78464 Konstanz, Germany}
	\address[B.~Rai\cb{t}\u{a}]{Scuola Normale Superiore, Centro di Ricerca Matematica Ennio De Giorgi, P.za dei Cavalieri, 3, 56126 Pisa, Italy
	}
	\address[J.~Van Schaftingen]{Universit\'{e} catholique de Louvain, Institut de Recherche en Mathématique et Physique, Chemin du Cyclotron 2 bte L7.01.02, 1348 Louvain-la-Neuve, Belgium}
	\makeatletter
	
	\@namedef{subjclassname@2020}{\textup{2020} Mathematics Subject Classification}
	
	\makeatother
	

	\begin{abstract}
		We identify necessary and sufficient conditions on $k$th order differential operators $\mathbb{A}$ in terms of a fixed halfspace $H^+\subset\mathbb{R}^n$ such that the Gagliardo--Nirenberg--Sobolev inequality
		$$
		\|D^{k-1}u\|_{\mathrm{L}^{\frac{n}{n-1}}(H^+)}\leq c\|\mathbb{A} u\|_{\mathrm{L}^1(H^+)}\quad\text{for }u\in\mathrm{C}^\infty_c (\mathbb{R}^{n},V)
		$$
		holds. This comes as a consequence of sharp trace theorems on $H=\partial H^+$.
	\end{abstract}
	\maketitle

	
	\section{Introduction}
	Let $\mathbb{A}$ be a homogeneous, linear, vectorial \(k\)th order partial differential operator with constant  coefficients between the finite dimensional real inner product spaces $V$ and $W$. That is, $\A$ has a  representation
	\begin{align}\label{eq:A}
		\A u(x)&=\sum_{|\alpha|=k} A_\alpha \partial^\alpha u(x),\qquad u\in\hold_{c}^{\infty}(\R^{n},V),
	\end{align}
	where $A_\alpha\in \lin(V,W)$ for $\alpha\in\N_0^n$, $|\alpha|=k$. In this context, a classical question is the validity of the estimate
	\begin{equation}
		\label{eq_queFaengeez3aeg5euYaeD3b}
		\|D^k u\|_{\lebe^p(H^{+})}\leq c\|\A u\|_{\lebe^p(H^{+})}\quad \text{for }u\in\hold_{c}^\infty(\overline{H^{+}},V)
	\end{equation}
	for some given open halfspace $H^{+}\subset\R^{n}$ with a constant $c>0$ independent of $u$. A positive answer implies that replacing in the Sobolev norm the full derivative \(D^k u\) by \(\A u\) is an equivalent norm on the Sobolev space, which is well adapted to problems in the Calculus of Variations and partial differential equations involving \(\A\); most notably, such coercive inequalities lead to well-posedness theorems in non-linear elasticity or fluid mechanics, see e.g. \cite{Friedrichs,FuchsSeregin,Korn}. By standard techniques such as flattening the boundary, one can then reduce the corresponding estimates on smooth domains and for operators with variable coefficients to the halfspace case.
	
	When \(1 < p < \infty\), the cases where \eqref{eq_queFaengeez3aeg5euYaeD3b} holds can be characterized following Aronszajn \cite[Thm.~V]{Aronszajn}, in terms of the symbol of \(\A\) defined as
	\begin{align*}
		\begin{split}
			\A(\xi)&=\sum_{|\alpha|=k} \xi^\alpha A_\alpha,\quad\text{for }\xi\in\R^n,
		\end{split}
	\end{align*}
	as follows: Estimate \eqref{eq_queFaengeez3aeg5euYaeD3b} holds if and only if both of the following conditions hold:
	\begin{enumerate}
		\item\label{itm:ell} $\A$ is \emph{elliptic} (in the interior), i.e., $\ker_\R\A(\xi)=\{0\}$ for all   $\xi\in\R^n\setminus\{0\}$;
		\item\label{itm:bell} $\A$ is \emph{boundary elliptic}, by which we mean that we have $\ker_{\mathbb{C}}\A(\xi+\imag \nu)=\{0\}$ for all $\xi\in\R^{n}$; here, $\nu$ denotes a unit normal to the hyperplane $H:=\partial H^{+}$.
	\end{enumerate}
	
	Different from the Calder\'{o}n--Zygmund theory  \cite{CalderonZygmund,CalderonZygmund1} for $1<p<\infty$, Ornstein has shown that when $p=1$, then \eqref{eq_queFaengeez3aeg5euYaeD3b} does  not hold unless one has the trivial \emph{pointwise} estimate \(\lvert D^k u (x)\rvert  \leq c \lvert \A u (x)\rvert\) \cite{Ornstein,KK,FG}. This obstruction is also referred to as \emph{Ornstein's non-inequality} and we  emphasize that there is no effect of the boundary here: The condition is necessary for \eqref{eq_queFaengeez3aeg5euYaeD3b} to hold for compactly supported functions. Similar results hold for $p=\infty$ \cite{Mityagin,dLM}.
	
	Consequently, strong $\lebe^1$ estimates, if possible, must bound weaker derivatives. In the case of interior estimates, building on the fundamental work of Bourgain--Brezis \cite{BB03,BB07,VS_div} and its higher-order generalization \cite{VS_jems}, it was shown in \cite{VS} that a Gagliardo--Nirenberg inequality
	\begin{align}\label{eq:VS}
		\|D^{k-1} u\|_{\lebe^{\frac{n}{n-1}}(\R^n)}\leq c\|\A u\|_{\lebe^1(\R^n)} \quad\text{for }u\in\hold^\infty_c(\R^n,V)
	\end{align}
	holds if and only if the operator $\A$ is elliptic and the equation $\A u=\delta_0w_0$ has no solution for $w_0\in W\setminus\{0\}$. This additional assumption, termed \emph{cancellation}, can be expressed algebraically as
	\begin{align}\tag{C}\label{eq:canc}
		\bigcap_{\xi\in\mathbb{S}^{n-1}}\mathrm{im\,}\A(\xi)=\{0\}
	\end{align}
	via the Fourier transform. This cancellation condition plays an important role in endpoint estimates of Hardy type and into Lorentz, Besov and Triebel--Lizorkin spaces, see for instance \cite{R_tams,R_CR,RS,SVS,HS,Stolyarov_Linfty,DieGm, Stolyarov,VS_pams}.
	
	When \(1 < p < \infty\), the condition on the boundary \ref{itm:bell} is a Lopatinski\u{\i}--Shapiro or covering condition \cite{Lopatinskii}. Such conditions were used successfully by Agmon--Douglis--Nirenberg and H\"ormander, among many others, to provide a satisfactory theory for determined and overdetermined elliptic systems \cite{ADN1,ADN2,Hormander}. 
	
	When \(p = 1\), estimates were established from different perspectives in \cite{Denk}, and in \cite{BrVS} for Poisson's equation with divergence free data. In \cite{GR}, it was shown that if \(\Omega \subset \R^n\) is a bounded domain with a smooth boundary \(\partial \Omega\), then the inequality
	\begin{align*}
		\|D^{k-1}u\|_{\lebe^{\frac{n}{n-1}}(\Omega)}\leq c\left(\|\A u\|_{\lebe^1(\Omega)}+\|u\|_{\lebe^1(\Omega)}\right)\quad \text{for }u\in\hold^\infty(\overline\Omega,V)
	\end{align*}
	is equivalent with boundary ellipticity of $\A$ in \textbf{all} directions, $$
	\ker_{\mathbb{C}}\A(\xi)=\{0\}\quad\text{for all }\xi\in\mathbb{C}^n\setminus\{0\}.$$
	We say that these operators are \emph{$\mathbb{C}$-elliptic}. This condition also plays a crucial role in establishing trace estimates in Lebesgue or Besov spaces \cite{BDG,DieGme21,GRVS}. However, little is
	known towards a comprehensive theory of global estimates for elliptic boundary value
	problems with $\lebe^1$ data.
	
	In view of the above discussion, the present paper gives a complete answer to the question of proving the Sobolev analogue of Aronszajn's result \eqref{eq_queFaengeez3aeg5euYaeD3b} for $p=1$. In the following, we denote   for $\nu\in\mathbb{S}^{n-1}$
	\begin{align}\label{eq:H_nu}
		H_{\nu}\coloneqq\{x\in\R^{n}\colon x\cdot\nu=0\}\quad\text{and}\quad
		H_{\nu}^{\pm}\coloneqq\{x\in\R^{n}\colon\mathrm{sgn}(x\cdot\nu)=\pm 1\}
	\end{align}
	the hyperplane with normal $\nu$ together with the corresponding adjacent halfspaces
	and we note that, in this terminology, $\nu$ is the inward unit normal to $\partial H_{\nu}^{+}$. The above classification problem is solved by the following theorem, displaying the first main result of the present paper: 
	\begin{theorem}\label{thm:main}
		Let \(n \ge 2\) and $\A$ be a $k$th order differential operator as in \eqref{eq:A}. Then the  following are equivalent:
		\begin{enumerate}
			\item\label{itm:main_a} The operator $\A$ is \emph{elliptic} (i.e., $\ker_\R\A(\xi)=\{0\}$ for $\xi\in\R^n\setminus\{0\}$) and \emph{boundary elliptic} in direction $\nu\in\mathbb S^{n-1}$ (i.e., $\ker_{\mathbb{C}}\A(\xi+\imag\nu)=\{0\}$ for all $\xi\in \R^{n}$).
			\item\label{itm:main_b} There exists a constant $c>0$ such that the Sobolev estimate
			\begin{align*}
				\|D^{k-1}u\|_{\lebe^{\frac{n}{n-1}}(H^+_\nu)}\leq c\|\A u\|_{\lebe^1(H^+_\nu)}
			\end{align*}
			holds for all $u\in\hold^\infty_c(\R^n,V)$.
		\end{enumerate}
	\end{theorem}
	To prove Theorem~\ref{thm:main}, we establish sharp trace theorems \emph{on given halfspaces} which enable us to extend to full space. Recall that the trace space of ${\dot\sobo}{^{k,1}}(H^+_\nu)$ is $\lebe^1(H_\nu)$ \cite{Gag} for $k=1$ and the Besov space ${\dot\besov}{^{k-1}_{1,1}}(H_\nu)$ for $k\geq 2$ \cite{Uspenskii}. It is therefore natural to split our analysis into first and higher order operators. For $k=1$, we have:
	\begin{theorem}\label{thm:trace_k=1}
		Let \(n \ge 2\) and $\A$ be a differential operator as in \eqref{eq:A} with $k=1$. Then the following are equivalent:
		\begin{enumerate}
			\item\label{itm:k=1_a} The operator $\A$ is boundary elliptic  in direction $\nu\in\mathbb S^{n-1}$.
			\item\label{itm:k=1_b} There exists a constant $c>0$ such that the trace estimate
			\begin{align}\label{eq:tracestimateThmMain1}
				\|u\|_{\lebe^1(H_\nu)}\leq c\|\A u\|_{\lebe^1(H_\nu^+)}
			\end{align}
			holds for all $u\in\hold^\infty_c(\R^n,V)$.
		\end{enumerate}
	\end{theorem}
	Even for operators with simple structure, estimating {differences} to arrive at \eqref{eq:tracestimateThmMain1} a la Gagliardo \cite{Gag} turns out to be hard as beyond the usual gradients this has only been achieved,  to the best of our knowledge, for symmetric gradients \cite{StrangTemam,Babadjian}. As for all other examples in which trace estimates are known \cite{BDG,DieGme21,GRVS}, the methods make instrumental use the fact that the differential operators concerned are boundary elliptic in every direction. Therefore, there is no hope to make these approaches work in the sharp case of Theorem \ref{thm:trace_k=1}.

	Our proof of the trace inequality in Theorem~\ref{thm:trace_k=1} uses an improved version of Smith's integral representation formula from \cite{Smith}, see Theorem~\ref{thm:smith} below. Crucially using the homogeneity and regularity properties of the underlying integral kernels, this brings us in a position to employ a similar argument as in Gagliardo's original proof of the trace inequality for $\sobo^{1,1}$-maps  \cite{Gag}, see Proposition \ref{prop:trace}\ref{itm:s=1}. This $\lebe^1$-estimate, however, is insufficient to prove the optimal higher order trace inequalities, which require Besov space estimates; see Proposition \ref{prop:trace}\ref{itm:s>1}. In this regard, we will prove the following:
	\begin{theorem}\label{thm:trace_k>1}
		Let \(n \ge 2\) and $\A$ be a differential operator as in \eqref{eq:A}  of order $k\geq 2$. Then the  following are equivalent:
		\begin{enumerate}
			\item\label{itm:k>1_a} The operator $\A$ is boundary elliptic  in direction $\nu\in\mathbb S^{n-1}$.
			\item\label{itm:k>1_b} There exists a constant $c>0$ such that the estimate
			\begin{align*}
				\|\partial^{k-1}_\nu u\|_{\lebe^1(H_\nu)}+\sum_{j=0}^{k-2}\|\partial_\nu^j u\|_{\dot{\besov}{^{k-1-j}_{1,1}}(H_\nu)}\leq c\|\A u\|_{\lebe^1(H_\nu^+)}
			\end{align*}
			holds for all $u\in\hold^\infty_c(\R^n,V)$.
		\end{enumerate}
	\end{theorem}
	The proof of this result hinges on refined Besov estimates on the integral kernels derived in Section~\ref{sec:prel}. This not only yields the sharp trace theorem for boundary elliptic operators but also displays a new method for the usual $k$th order gradient case. Moreover, from a conceptual  perspective of limiting $\lebe^{1}$-estimates involving differential operators, the proof of Theorem \ref{thm:trace_k>1} seems to be the first approach that systematically uses difference estimates despite the lack of the $\lebe^{1}$-control of full $k$th order gradients due to  Ornstein's non-inequality.
	
	Note that in fact Theorems \ref{thm:trace_k=1} and \ref{thm:trace_k>1} lead to a unified treatment of first and higher order operators by considering the space
	$$
	\mathrm{T}_k(H_\nu ,V) \coloneqq \left\{(f_0,f_1,\ldots,f_{k-1})\colon \begin{array}{c} f_{k-1}\in\lebe^1(H_\nu ,V),\\ f_j\in\dot{\besov}{^{k-1-j}_{1,1}}(H_\nu ,V),\,0\leq j\leq k-2 \end{array}\right\},
	$$
	endowed with the obvious norm, cf. Theorem \ref{thm:trace_unified}. We will show in Section \ref{sec:BVA} that for $k$th order operators $\A$, the boundary ellipticity condition in direction $\nu$ is equivalent to the fact that $\tr_k(\bv^\A(H_\nu^+))=\mathrm T_k(H_\nu,V)$, where the trace operator is defined as
	$$
	\tr_k u\coloneqq(u,\partial_\nu u,\ldots,\partial_\nu^{k-1}u)\big|_{H_\nu}
	$$
	for functions smooth up to the boundary and then extended by continuity for a suitable sort of strict convergence; see \eqref{eq:funcspacesDef}ff. for the underlying terminology of such generalized $\bv$-type spaces. We show that this trace map admits a continuous right inverse that cannot be linear, generalizing Peetre's Theorem, cf. \cite{peetre,PW1,PW2}. These facts seem to have gone unnoticed also in the basic case of $\bv^k(\R^n_+)$ and $\dot{\sobo}{^{k,1}}(\R^n_+)$.
	
	To prove Theorem~\ref{thm:main}, we  use the unified extension theorem in  $\dot\sobo{^{k,1}}(H_\nu^-)$ (see Theorem \ref{thm:extension_many_derivatives}) to reduce the question to an estimate in full space, in this case \eqref{eq:VS}. 
	To see if this estimate is available, we should combine the canceling condition and boundary ellipticity. Interestingly, it turns out that the boundary ellipticity of  \ref{itm:main_a} implies the canceling condition \eqref{eq:canc}, see Proposition~\ref{prop:bc_ec}.
	
	Using the same extension procedure, we can prove versions up to the boundary of other estimates that are known in full space \cite{VS,BVS,R_tams}, see Theorem \ref{thm:estimates}. This particularly allows us to bound all fractional derivatives of $D^{k-1}u$ (e.g. in the Sobolev-Slobodecki\u{\i} scale) \emph{up to, but not including}, order one on halfspaces against $\A u$. In light of Ornstein's negative result, this  displays the optimal generalisation of Aronszajn's result to the case $p=1$. 
	
	Finally, let us point out that all of the preceding theorems admit interpretations in the potential theory  for elliptic systems on halfspaces with $\lebe^{1}$-data; Theorem~\ref{thm:main} then corresponds to the case of elliptic systems with identically null boundary conditions. Whereas the focus of the present paper is on the generalisation of Aronszajn's result to the case $p=1$, it may also be seen as a first step towards a comprehensive theory of $\lebe^{1}$-estimates for general boundary value problems. We will pursue this in later work.
	
	The paper is organized as follows: in Section~\ref{sec:prel} we give a comprehensive trace and extension theory in $\dot{\sobo}{^{k,1}}(\R^n_+)$ and  establish the improvement of the representation formula from Smith's work that will be instrumental for the main results. In Section \ref{sec:convolution} we prove the main trace estimates on convolution operators. In Section~\ref{sec:trace} we prove both trace theorems, and in Section \ref{sec:sob} we establish the Sobolev estimate and its extensions. In Section \ref{sec:BVA} we extend our estimates to spaces of rough functions and comment on the notion of boundary ellipticity.
	{\small 
		\subsection*{Acknowledgement}
		We thank Petru Mironescu for suggesting helpful references. F.G. also gratefully acknowledges financial support through the Hector foundation FP, Project 626/21. }

	\section{Traces for $\sobo^{k,1}(\R_{+}^{n})$-maps and representation formulas}\label{sec:prel}
	In this section, we revisit the trace theory for functions in the Sobolev space $\sobo^{k,1}$ on halfspaces and give an improved variant of Smith's representation formula to play a crucial role in the subsequent sections.
	
	In view of our main results, we will assume that we work in space dimensions $n>1$ throughout. The reason for this is that, for $n=1$, the only relevant operator is $\A(t)=t^k$, in which case we have the embedding $\dot\sobo{^{k,1}}(\R_+)\hookrightarrow \hold_0^{k-1}(\R_+)$, where the latter space is endowed with the norm $u\mapsto\|u^{(k-1)}\|_{\infty}$. In this case, Theorem \ref{thm:main} holds with the convention $1/0=\infty$ and all derivatives up to order $k-1$ have well defined point values at $0$.
	\subsection{Traces for the Sobolev space ${\sobo}{^{k,1}}(\R^n_+)$}
	The results below necessitate some background facts from Besov space theory, and we refer the reader to Triebel's encyclopedic monograph  \cite[\S 5]{Triebel} for the definition and elementary properties of homogeneous Besov spaces. Most importantly for us, 
	we require a characterisation of homogeneous Besov spaces in terms of finite differences \cite[\S 5.2.3, Thm. 1]{Triebel} that we record explicitely: Given $k\in\mathbb{N}$ and a map $u\colon\R^{m}\to V$, we put for $h\in\R^{m}$
	\begin{align*}
		\Delta_{h}^{k}u(x):=\sum_{i=0}^{k}(-1)^{k-i}\binom{k}{i}u(x+ih),\qquad x\in\R^{m}. 
	\end{align*}
	Moreover, given  $s>0$ and $k\in\mathbb{N}$ with $k>s$, we define the seminorms
	\begin{align}\label{eq:seminormBesov}
		\|u\|_{{\dot\besov}{_{1,1}^{s}}(\R^{m})}\coloneqq\int_{\R^{m}}\int_{\R^{m}}|\Delta_{h}^{k}u(x)|\dif x\frac{\dif h}{|h|^{m+s}}
	\end{align}
	for $u\in\lebe_{\locc}^{1}(\R^{m},V)$, and it is clear that \eqref{eq:seminormBesov} defines a norm on $\hold_{c}^{\infty}(\R^{m},V)$. We note that for any two such choices of $k>s$ the corresponding seminorms on the right-hand side of~\eqref{eq:seminormBesov} are equivalent and, in particular, define equivalent norms on ${\dot{\besov}}{_{1,1}^{s}}(\R^{m},V)$. Upon  tacitly identifying boundaries $H_{\nu}$ of half-spaces $H_{\nu}^{\pm}$ with $\R^{n-1}$, all of the preceding notions carry over to functions defined on $H_{\nu}$.

	Based on these definitions and identifying $\R^{n-1}\simeq\R^{n-1}\times\{0\}$, the classical results of Gagliardo \cite{Gag} and Uspenski\u{\i} \cite{Uspenskii} (also see \cite{MR,Mironescu_2015}) can be stated as follows:
	\begin{theorem}\label{thm:gag+usp}
		Let $k\geq 2$. Then we have the trace inequalities
		$$
		\|u\|_{\lebe^1(\R^{n-1})}\leq c\|Du\|_{\lebe^1(\R^n_+)}\quad\text{and}\quad\| u\|_{\dot{\besov}{^{k-1}_{1,1}}(\R^{n-1})}\leq c\|D^ku\|_{\lebe^1(\R^n_+)}
		$$
		for all $u\in\hold^\infty_c(\R^n)$.
	\end{theorem}
	As a consequence, one can immediately prove the estimate of Theorem \ref{thm:trace_k>1} in the case of Sobolev spaces, $\A=D^k$:
	$$
	\|\tr_k u\|_{\mathrm T_k(\R^{n-1})}\leq c\|D^k u\|_{\lebe^1(\R^n_+)} \quad\text{ for all }u\in\hold^\infty_c(\R^n), 
	$$
	where $$
	\tr_k u\coloneqq(u,\partial_n u,\ldots,\partial_n^{k-1}u)\big|_{\R^{n-1}}.
	$$
	The trace operators that can be defined by the estimates of Theorem \ref{thm:gag+usp} admit continuous right inverses, but these are insufficient for our purposes. We will prove the following extension theorem, which is probably known to the experts but seems absent from the literature:
	\begin{theorem}\label{thm:extension_many_derivatives}
		There exists a constant \(c>0\) such that for all \(g_0, \dotsc, g_{k - 1} \in \hold^\infty_c (\R^{n - 1})\), there exists  \(u \in \hold^\infty (\R^n_+)\) such that for \(j = 0, \dotsc, k - 1\)
		\begin{equation*}
			\partial_n^j u (\,\cdot\,, 0) = g_j
		\end{equation*}
		and 
		\begin{equation*}
			\lVert D^k u \rVert_{\lebe^1 (\R^n)}
			\leq c\left( \lVert g_0 \rVert_{\dot{\besov}{^{k-1}_{1,1}}(\R^{n - 1})}
			+ \dotsb 
			+ \lVert g_{k - 2} \rVert_{\dot{\besov}{^{1}_{1,1}}(\R^{n - 1})} 
			+ \lVert g_{k - 1} \rVert_{\lebe^1 (\R^{n - 1})}\right). 
		\end{equation*}
	\end{theorem}
	The proof of this result is done in two steps: First we construct extension operators that satisfy each Dirichlet condition separately. Then we use a superposition formula to put these extensions together. 
	\begin{proposition}
		\label{proposition_extension_besov}
		For every \(j \in \{0, \dotsc, k - 2\}\), if \(g_j \in \dot{\besov}{^{k - j}_{1,1}}(\R^{n - 1})\), there exists \(u \in \hold^\infty (\R^n_+)\) such that \(\partial_n^j u (\,\cdot\,, 0) = g_j\) and 
		\begin{equation*}
			\lVert D^k u \rVert_{\lebe^1 (\R^n)}
			\leq c\lVert g_j \rVert_{\dot{\besov}{^{k-j - 1}_{1,1}}(\R^{n - 1})}. 
		\end{equation*}
	\end{proposition}
	\begin{proof}
		By \cite{Uspenskii} and \cite[Thm.~1.2]{MR}, there exists \(v \in \hold^\infty (\R^n_+)\) such that \(v (\,\cdot\,, 0) = g\) and if \(\ell \leq j \),
		\begin{equation*}
			\int_{\R^n_+} x_n^{j - k} \vert D^{k - \ell} v (x', x_n)\vert \dif x
			\leq c \lVert g_j \rVert_{\dot{\besov}{^{k-j - 1}_{1,1}}(\R^{n - 1})}. 
		\end{equation*}
		Defining \(u (x',x_n) = x_n^j v (x', x_n)/j!\), we reach the conclusion. 
	\end{proof}
	
	\begin{proposition}[{\cite[Lem. 3.3]{PW1}}]
		\label{proposition_extension_lebe}
		If \(g \in \lebe^1 (\R^{n - 1})\), there exists \(u \in \hold^\infty (\R^n_+)\) such that \(\partial_n^{k - 1} u (\,\cdot\,, 0) = g\) and 
		\begin{equation*}
			\lVert D^k u \rVert_{\lebe^1 (\R^n)}
			\leq c\lVert g \rVert_{\lebe^1 (\R^{n - 1})}. 
		\end{equation*}
	\end{proposition}
	\begin{proof}
		The proof given by Mironescu \cite{Mironescu_2015} for \(p = 1\) has a straightforward adaptation. 
		Taking a function \(\theta \in \hold^\infty_c(\R)\) such that \(\theta (0) = 1\) and \(\theta '(0) = \dotsb = \theta^{(k - 1)} (0) = 0\), we define for $\varepsilon>0$ to be chosen later
		\begin{equation*}
			u (x', x_n) = \theta (x_n/\varepsilon) x_n^{k- 1} g (x').
		\end{equation*}
		We have 
		\begin{equation*}
			\int_{\R^n_+} \vert D^k u \vert \dif x
			\leq c \sum_{j = 0}^k \int_{\R^{n - 1}} \varepsilon^k \vert D^k g \vert\dif x';
		\end{equation*}
		taking \(\varepsilon > 0\) small enough, we reach the conclusion.
	\end{proof}
	
	\begin{proof}[Proof of Proposition~\ref{thm:extension_many_derivatives}]
		By Propositions \ref{proposition_extension_besov} and \ref{proposition_extension_lebe}, let \(u_j\) be given so that 
		\begin{equation*}
			\partial_n^j u_j (\,\cdot\,, 0) = g_j.
		\end{equation*}
		We apply now a linear superposition of dilations (see \cite[Thm. 2.2]{Lions_Magenes} and \cite[Thm. 4.26]{Adams}).
		Defining now 
		\begin{equation*}
			u (x', x_n) = \sum_{j = 0}^{k -1} \sum_{i = 1}^k \mu_{i, j} u_j (x',\lambda_i x_n),
		\end{equation*}
		with fixed distinct $\lambda_{1},...,\lambda_{k}\in(0,\infty)$ 
		under the condition on $\mu_{i,j}$ that 
		\begin{equation*}
			\sum_{i = 1}^k \mu_{i, j} \lambda_i^\ell  = \delta_{j, \ell},
		\end{equation*}
		we reach the conclusion.
	\end{proof}
	Finally, we remark that the  trace operators in ${\sobo}{^{k,1}}$, as defined by Theorem \ref{thm:gag+usp},  have continuous inverses, but these can only be linear for $k>1$ \cite{peetre,MR}. A generalization of Peetre's result that Gagliardo's trace operator $\tr ({\sobo}{^{1,1}})=\lebe^1$ cannot have a bounded linear inverse to the $k$th order Sobolev space is proved in \cite[Thm. 5.1]{PW2}. Here we prove the following related result:
	\begin{proposition}
		The bounded linear trace operator 
		\begin{align*}
			\tr_k \colon {\dot{\sobo}}{^{k,1}}(\R^n_+)\to \tr_{k}({\dot{\sobo}}{^{k,1}}(\R^n_+))=\mathrm T_k(\R^{n-1})
		\end{align*}
		does not admit a right inverse that is both linear and continuous.
	\end{proposition}
	\begin{proof}
		Suppose that $E$ is a bounded linear inverse of $\tr_k$ and let $f\in\lebe^1(\R^{n-1})$, so $(0,\ldots,0,f)\in\mathrm T_k(\R^{n-1})$. Write $u\coloneqq E(0,\ldots,0,f)\in {\dot{\sobo}}{^{k,1}}(\R^n_+)$ so that $\partial_n^{k-1}u\in{\dot{\sobo}}{^{1,1}}(\R^n_+)$. Note that then $f\mapsto \partial_n^{k-1}u$ is a bounded linear inverse of Gagliardo's trace operator. This contradicts Peetre's theorem.
	\end{proof}
	
	\subsection{The Smith integral representation}

	In this section  we revisit and improve Smith's construction of representation formulas implied by the boundary ellipticity condition
	\cite{Smith0,Smith}. Precisely, we have  
	\begin{theorem}\label{thm:smith}
		Let $\A$ as in \eqref{eq:A} be  boundary elliptic in direction $\nu\in\mathbb{S}^{n-1}$. Then there exists a $(k-n)$-homogeneous convolution kernel $K\in\hold^\infty(\R^n\setminus\{0\},\lin(W,V))$ such that $K=0$ in $H^-_\nu$ and
		\begin{align}\label{eq:smith}
			u(x)=\int_{H^+_\nu}K(y)\A u(x+y)\dif y\quad\text{for } x\in \R^n,
		\end{align}
		for all $u\in\hold^\infty_c(\R^n,V)$.
	\end{theorem} 
	We give a direct proof of Theorem \ref{thm:smith}, following the approach in \cite[Sec. 3]{Smith}, where our presentation yields \(\hold^\infty\)- instead of \(\hold^l\)-smoothness for fixed arbitrarily large \(l \in \N\). We start with a variant of the Sobolev integral representations in the spirit of \cite[Sec.~2]{Smith}; here, we use the notation $V\otimes\bigotimes^k\R^n$ to denote the space of $V$-valued $k$-linear maps on $\R^n$.
	
	\begin{proposition}\label{prop:sobolev}
		There exists a $(k-n)$-homogeneous convolution kernel  $K_k\in\hold^\infty(\R^n\setminus\{0\},\lin (V\otimes\bigotimes^k\R^n,V))$  such that for every \(x = (x', x_n) \in \R^n\) with \(x_n \leq \lvert x \rvert/2\), \(K_k (x) = 0\) and 
		\begin{align}\label{eq:sobolev}
			u(x)=\int_{\R^n_+}K_k(y)D^k u(x+y)\dif y,\quad\text{for }x\in\R^n,
		\end{align}
		for all $u\in\hold^\infty_c(\R^n,V)$.
		
		Moreover, for every \(r \in \N\),
		there exists a \( (k + r - n) \)-homogeneous convolution kernel \(K_k^r \in (\R^{n} \setminus \{0\}, \Lin (V \otimes \bigotimes^k \R^n \otimes \bigotimes^r \R^{n - 1}, V))\) such that 
		\begin{equation*}
			\int_{\R^{n - 1}} K_k (y', y_n) v (y') \dif y'
			= \int_{\R^n_+} K_{k}^r(y, y_n)D^r v(y')\dif y.
		\end{equation*}
		for every \(v \in\hold^\infty_c (\R^{n - 1}, V \otimes \bigotimes^k \R^n)\) and \(y_n \in (0, +\infty)\).
	\end{proposition}
	
	\begin{proof}
		By integration by parts we have that for $\theta\in\mathbb{S}^{n-1}$ it holds that
		\begin{align*}
			u(x)=c\int_0^\infty t^{k-1}\dfrac{\dif^k}{\dif t^k}u(x+t\theta)\dif t=c\int_0^\infty t^{k-1}D^ku(x+t\theta)\left[\theta^{\otimes k}\right]\dif t.
		\end{align*}
		We fix \(\varphi \in \hold^\infty ([-1, 1])\) such that \(\varphi = 0\) on \([-1, 1/2]\)
		and \(\int_{\mathbb{S}^{n - 1}} \varphi (\theta_n) \dif \theta = 1\).
		We then have 
		\begin{align*}
			u(x)&=c\int_{\mathbb{S}^{n-1}}\int_0^\infty t^{k-1}D^ku(x+t\theta)\left[\theta^{\otimes k}\right]\varphi(\theta n)\dif t\dif \theta\\
			&=c\int_{\R^n_+} |y|^{k-1}D^ku(x+y)\left[\left(\frac{y}{|y|}\right)^{\otimes k}\right]\varphi\left(\frac{y_n}{|y|}\right)\dfrac{\dif y}{|y|^{n-1}},
		\end{align*}
		which suffices to conclude the proof of the first statement.
		
		The second statement stems from the fact that if \(H (z) = z^{\otimes m} \eta (\vert z\vert)\), \(z \in \R^{n - 1} \times \{0\} \simeq \R^{n - 1}\), then \(\operatorname{div} H = (n + m) z^{\otimes m - 1} \eta (\vert z\vert) + z^{\otimes m} \eta' (\lvert z \rvert)\). The kernels \(K_k^r\) can then be computed recursively through the latter identity and have the required properties.  
	\end{proof}

	\begin{proof}[Proof of Theorem~\ref{thm:smith}]
		There is no loss of generality in setting $\nu=e_n$, i.e., $H_\nu=\R^{n-1}$ and $H^+_\nu=\R^n_+$. We will use coordinates $x=(x^\prime,x_n)$ (real or complex), defined in an obvious way. 
		
		We begin by assuming that \(\dim V = 1\), so that the boundary ellipticity assumption then reduces to 
		the condition \(\A (\xi + i \nu) \ne 0\) for every \(\xi \in \R^{n}\).
		In particular, we have 
		\[
		\{ \xi = (\xi', \xi_n) \in \C^n : \A (\xi) = 0 \text{ and } \xi' = 0 \}
		= \{0\},
		\]
		therefore by Hilbert's strong Nullstellensatz (see for example \cite[Chpt.\ 4, \S 2, Thm. 6]{CLO}), for \(d \in \N\) large enough, there exist \(\mathbb{M}_1\) and \(\mathbb{M}_2\) homogeneous differential operators of orders \(d - k\) and \(d - \ell\), such that 
		\[
		\xi^{\otimes d} = \mathbb{M}_1 (\xi) \A (\xi) + \mathbb{M}_2(\xi) \xi'{}^{\otimes\ell}.
		\]
		By the Sobolev representation formula (Proposition~\ref{prop:sobolev}) and by integration by parts, it follows that 
		\begin{align}
			u(x) & = \int_{\R^n_+} K_d (y) D^d u (x + y) \dif y \notag \\
			&=  \int_{\R^n_+} K_d (y) \mathbb{M}_1 \mathbb{A} u (x + y) \dif y
			+ \int_{\R^n_+} K_d (y) \mathbb{M}_2 D'^{\ell} u (x + y) \dif y\label{eq_ia2ohW2wuu3aWienoiCheigh} \\
			& = \int_{\R^n_+}  (\mathbb{M}_1^* K_d^*)^* (y) \mathbb{A} u (x + y) \dif y
			+(-1)^{\ell}\int_{\R^n_+}  (\diver'^{\ell} \mathbb{M}_2^* K_d^*)^* (y) u (x + y) \dif y,\notag
		\end{align}
		where \(K_d^* (y) \coloneqq (K_d(y))^*\) is the pointwise adjoint.
		We observe that operator \( (\mathbb{M}_1^* K_d^*)^* \) is homogeneous of degree \(k - n\) and that \((\mathbb{M}_1^* K_d^*)^* (x) = 0\) if \(x_n \leq \vert x\vert/2\), so that the first term in the right-hand side of \eqref{eq_ia2ohW2wuu3aWienoiCheigh} has the required structure.
		
		Since the operator \(\A\) is boundary elliptic, we have \(w_0 = \A (0, 1) \ne 0\).
		By homogeneity of the operator \(\A\), it follows that for every \(\xi' \in \R^{n - 1}\), \(w_0 \cdot \A(\xi',\,\cdot\,)\) is a polynomial of degree \(k\), its leading order term being \(\vert w_0 \vert^2 \xi_n^k\).

		We define \(\mathbb{P}(\xi')\) as a differential operator from \(\R\) to the space \(\R[W]_k\) of homogeneous polynomials on \(W\) of degree \(k\) by \(\mathbb{P} (\xi')[w] = \operatorname{Res} (w_0 \cdot \A (\xi),  w \cdot \A (\xi), \xi_n)\), the resultant of the polynomials \(w_0 \cdot \A (\xi)\) and \(w \cdot \A (\xi)\), seen as polynomials in \(\xi_n\) over the ring of polynomials in \((\xi', w)\)  (see for example \cite[Chpt.\ 3,  \S\S 5--6]{CLO}).
		Given \(\xi' \in \R^{n - 1}\), we let
		\(\tau_1, \dotsc, \tau_k \in \C\) be the roots of the polynomial \(w_0 \cdot \A (\xi', \,\cdot\,)\) and we define the linear subspaces
		\begin{equation*}
			W_0 = \{w \in W : \deg (w \cdot \mathbb{A} (\xi', \,\cdot\,)) < k\}
		\end{equation*}
		and 
		\begin{equation*}
			W_j = \{w \in W : w \cdot \mathbb{A}(\xi', \tau_j) = 0\}, \quad  j \in \{1, \dotsc, k\}.
		\end{equation*}
		By boundary ellipticity, we have \(W_j \ne W\) for every \(j \in \{0, \dotsc, k\}\).
		By the properties of resultants, we have 
		\begin{equation*}
			\{ w \in W : \mathbb{P} (\xi')[w]=0\} \subseteq \bigcup_{j = 0}^k W_j \ne W,
		\end{equation*}
		and thus \(\mathbb{P} (\xi') \ne 0\). Since \(\mathbb{P}\) is a differential operator on scalar functions, this is equivalent to having \(\mathbb{P}\) elliptic.
		
		From the definition of the resultant \(\mathbb{P}\) as a Sylvester determinant, \(\mathbb{P}\) is homogeneous of degree \(k^2\) and there exist homogeneous differential operator \(\Q\) of order \(k (k - 1)\) from \(W\) into \(\R[W]_k\) such that
		\[
		\mathbb{P} (\xi') = \mathbb{Q} (\xi) \A (\xi).
		\]
		By ellipticity, the operator $\mathbb{P}$ on $\R^{n-1}$ defined by $\mathbb{P} (\xi^\prime)$ has a fundamental solution $E\in\hold^\infty(\R^{n - 1} \setminus\{0\}, \Lin(\R[W]_k, \R))$, i.e.,
		\begin{equation}
			\label{eq_ood4xaek6pas7eijieGu8soh}
			v(x^\prime)
			=
			\int_{\R^{n-1}}E(y^\prime)\mathbb{P} v(x^\prime-y^\prime)\dif y^\prime,\quad\text{for }x'\in\R^{n-1},\,v\in\hold^\infty_c(\R^{n-1}),
		\end{equation}
		which is such that $D^s E$ is $(k^2  -(n-1)-s)$-homogeneous for \(s \in \N\) provided that either \(n = 2\) (when one is inverting a differential operator on \(\R^{n - 1}\) by the fundamental theorem of calculus) or when $s \ge k^2 - (n - 1) + 1$ (see, for instance \cite[Chpt.~VII]{HormI}).
		
		We rewrite the second term of the right-hand side of \eqref{eq_ia2ohW2wuu3aWienoiCheigh}
		thanks to \eqref{eq_ood4xaek6pas7eijieGu8soh} as 	
		\begin{equation}
			\label{eq_fah7pho1OoxieBeiMei3quee}
			\begin{split}
				\int_{\R^n} & (\diver'^\ell \mathbb{M}_2^* K_d^*)^* (y-x)  u(y)\dif y\\
				&=\int_{\R^n}(\diver'^\ell \mathbb{M}_2^* K_d^*)^* (y-x) \biggl(\int_{\R^{n-1}} E(y^\prime-z^\prime)  \mathbb{P} u(z^\prime,y_n) \dif z^\prime\biggr) \dif\,( y^\prime, y_n)\\
				&= \int_{\R^n}\biggl(\int_{\R^{n-1}}(\diver'^\ell \mathbb{M}_2^* K_d^*)^* (y-x) E(y^\prime-z^\prime) \dif y^\prime\biggr) \, \mathbb{P} u(z^\prime,y_n) \dif\,( z^\prime, y_n)\\
				& =  \int_{\R^n} 
				F (x^\prime-z^\prime,y_n-x_n) \mathbb{P} u(z^\prime,y_n) \dif\,( z^\prime, y_n),
			\end{split}
		\end{equation}
		where 
		\begin{align*}
			F (y', y_n)&=\int_{\R^{n-1}}(\diver'^\ell \mathbb{M}_2^* K_d^*)^*(v^\prime-y^\prime,y_n) E(v^\prime) \dif y^\prime\\
			&=\int_{\R^{n-1}}(\diver'^\ell \mathbb{M}_2^* K_d^*)^* (w^\prime, y_n) E(w^\prime+y^\prime) \dif w^\prime.
		\end{align*}
		Since \((\mathbb{M}_2^* K_d^*)^* (w^\prime, y_n) = 0\) whenever \(\vert w^\prime\vert \ge 2 y_n\) and since \(\mathbb{B} E\) is locally integrable provided \(\ell \leq k^2 -1\).
		If we assume moreover, that we have chosen \(\ell = k^2 - 1\), 
		then \(D'^\ell E\) is homogeneous of degree \(2 - n\) and thus, by a suitable integration by parts, we get that \(F\) is homogeneous of degree \(k - n\).
		
		
		We take $\varphi\in\hold^\infty_c(\R^n)$ such that $\varphi=1$ in a neighbourhood of $0$ and $\varphi=0$ outside the ball of radius $|x^\prime|/2$. In view of \eqref{eq_fah7pho1OoxieBeiMei3quee}, we have that, for $(y^\prime,y_n)$ near $(x^\prime,0)$, for every \(r \in \N\), 
		\begin{equation}
			\label{eq_ohkaZae2aaxosie4ahph1eiV}
			\begin{split}
				F (y^\prime,y_n)=&\int_{\R^{n-1}}(\operatorname{div}'^\ell \mathbb{M}_2^* K_d^*)^*(w^\prime,y_n)\varphi (w^\prime+y^\prime)E(w^\prime+y^\prime)\dif w^\prime\\
				&+\int_{\R^{n-1}}  K_d^r (w^\prime,y_n) D'^r \mathbb{M}_2^* D'^\ell \left[(1-\varphi)E\right](w^\prime+y^\prime)\dif w^\prime.
			\end{split}		
		\end{equation} 
		The first integral defines a function that is of class $\hold^\infty$ in a neighbourhood of $(x^\prime,0)$; the second integral can be differentiated \(d + r - 1\) times without destroying the integrability, it thus follows \(F\) is of class \(\hold^{d + r - 1}\) in a neighbourhood of \((x', 0)\) with arbitrary \(r \in \N\).
		
		Taking \(K = (\mathbb{M}_1^* K_d^*)^*  + F\), we reach the conclusion when \(\dim V = 1\).
		
		\medbreak
		
		If \(\dim V = m \ge 2\), we consider the operator \(\medwedge^m \A\) of order \(mk\) from \(\medwedge^m V\) to \(\medwedge^m W\) defined for \(v_1, \dotsc, v_m \in V\) by
		\begin{equation*}
			\medwedge^m \A(\xi) (v_1 \wedge \dotsb \wedge v_m) 
			=
			(\A (\xi)v_1) \wedge \dotsb \wedge (\A (\xi) v_m).
		\end{equation*}
		There exists an operator \(\mathbb{S}\) of order \((m - 1)k\) from \(W\) to \(\Lin(\medwedge^{m - 1} V, \medwedge^{m} W)\) such that for every \(\omega \in \medwedge^{m - 1} V\),
		\begin{equation*}
			\medwedge^m \A (\omega \wedge v) = (\mathbb{S} \A v) (\omega). 
		\end{equation*}
		Moreover, letting \(v_1, \dotsc, v_m\) be a basis of \(V\) and choosing \(\omega_1, \dotsc, \omega_m \in \medwedge^{m - 1} V\) such that \(\omega_i \wedge v_j = \delta_{ij}\), so that for every \(v \in V\), one has 
		\begin{equation*}
			v = \sum_{j = 1}^m (\omega_i \wedge v) v_i
		\end{equation*}
		(upon identification between \(\medwedge^m V\) and \(\R\)).
		Letting \(\Hat{K}\) be the homogeneous kernel of order \(mk - n\) given for \(\medwedge^m \A\) in the first part of the proof, we have the identity
		\begin{equation*}
			\begin{split}
				u (x) &= \sum_{i = 1}^m  (\omega_i \wedge u (x)) v_i\\
				&= \sum_{i = 1}^m \int_{\R^n_+} K (y) \medwedge^m \A (\omega_i \wedge u) (x + y) \dif y v_i\\
				&= \sum_{i = 1}^m \int_{\R^n_+} K (y)  \mathbb{S} \A u (x + y) (\omega_i)\dif y v_i\\
				&= \int_{\R^n_+} \bigl(\sum_{i = 1}^m \mathbb{S}[\omega_i]^* K (y) v_i\bigr)  \A u (x + y)\dif y,
			\end{split}
		\end{equation*}
		which is the conclusion.
	\end{proof}
	
	\section{Estimates on special convolution operators}\label{sec:convolution}
	The main analytical advancement of this paper is contained in the following estimate for convolution operators with kernels that vanish on a halfspace; in the sequel, we identify $\R^{n-1}\times\{0\}\simeq\R^{n-1}$.
	\begin{proposition}\label{prop:trace}
		Let $s\geq 1$ and $K\in\hold^\infty(\R^n\setminus\{0\})$ be $(s-n)$-homogeneous and satisfy $K\equiv 0$ in $\R^n_-$. Let
		$$
		\mathcal T f(x)\coloneqq \int_{\R^n_+}K(y)f(x+y)\dif y\quad\text{for }f\in\hold^\infty_c(\R^n).
		$$
		Then we have the estimates:
		\begin{enumerate}
			\item\label{itm:s=1} if $s=1$
			$$
			\|\mathcal T f\|_{\lebe^1(\R^{n-1})}\leq c\|f\|_{\lebe^1(\R^n_+)}\quad\text{for }f\in\hold^\infty_c(\R^n),
			$$
			\item\label{itm:s>1} if $s>1$
			$$
			\|\mathcal T f\|_{{\dot{\besov}}{^{s-1}_{1,1}}(\R^{n-1})}\leq c\|f\|_{\lebe^1(\R^n_+)}\quad\text{for }f\in\hold^\infty_c(\R^n).
			$$
		\end{enumerate}
	\end{proposition}
	We will use coordinates $x=(x',t)$, $x'\in\R^{n-1}$. The proof consists by first noticing that it suffices to show that $K(\,\cdot\,,1)\in \lebe^1$ (resp. ${\dot{\besov}}{^{s-1}_{1,1}}$) if $s=1$ (resp. $s>1$), followed by checking these claims using the fact that the smoothness and vanishing of $K$ on $\R^n_-$ give us better bounds than homogeneity alone.
	\begin{proof}
		We begin with the proof of \ref{itm:s=1}, the case $s=1$. We first show that it suffices to prove that $K(\,\cdot\,,1)\in\lebe^1 (\R^{n-1})$. 
		
		We have that by a simple change of variable
		\begin{align*}
			\mathcal{T}f(x',0)&=\int_{\R^{n}_+}K(y)f((x',0)+y)\dif y=\int_{(x',0)+\R^{n}_+}K(z-(x',0))f(z)\dif z\\
			&=\int_{\R^{n}_+}K(z-(x',0))f(z)\dif z,
		\end{align*}
		so that by Fubini's theorem
		\begin{align*}
			\int_{\R^{n-1}} |\mathcal{T}f(x',0)|\dif x' &\leq \int_{\R^{n-1}}\int_{\R^n_+}|K(z-(x',0))||f(z)|\dif z   \dif x'\\
			&=\int_{\R^n_+}\left( \int_{\R^{n-1}}|K(z-(x',0))|\dif x'\right) |f(z)|   \dif z.
		\end{align*}
		We will show that the inner integral is independent of $z=(z',t)\in\R^n_+$. We make the change of variable $y'=t^{-1}(z'-x')$ to get
		\begin{align*}
			\int_{\R^{n-1}}|K(z-(x',0))|\dif x'&=\int_{\R^{n-1}}|K(z'-x',t)|\dif x'\\
			&=\int_{\R^{n-1}}t^{1-n}|K(t^{-1}(z'-x'),1)|\dif x'=\int_{\R^{n-1}}|K(y',1)|\dif y'.
		\end{align*}
		Therefore 
		\begin{align}\label{eq:L1forK}
			\int_{\R^{n-1}} |\mathcal{T}f(x',0)|\dif x' &\leq \int_{\R^{n-1}}|K(y',1)|\dif y' \int_{\R^n_+} |f(z)|\dif z.
		\end{align}
		To show that \eqref{eq:L1forK} implies the estimate in \ref{itm:s=1}, we will prove that $K(\,\cdot\,,1)\in\lebe^1 (\R^{n-1})$.
		
		To achieve this, we fix $0<\alpha<1$ and will only use the fact that $K\in\hold^{0,\alpha}_{\locc}(\R^n\setminus\{0\})$ together with a homogeneity argument. We write $\mathbb{S}^{n-2}=\R^{n-1}\cap\mathbb{S}^{n-1}$ and denote, for an arbitrary but fixed $0<r<1$, the neighbourhood $\mathcal{N}\coloneqq\mathbb{S}^{n-2}+B_r(0)$ of $\mathbb{S}^{n-2}$ in $\R^n$. We then define $c_{1}$ to be the $\alpha$-H\"{o}lder seminorm of $K$ on $\mathcal{N}$, i.e., 
		\begin{align*}
			c_{1}\coloneqq\sup_{\substack{x,y\in\mathcal{N},\\ x\neq y}}\frac{|K(x)-K(y)|}{|x-y|^{\alpha}}.
		\end{align*}
		\begin{figure}
			\begin{tikzpicture}[scale=0.85]
				\draw [black!10!white, thick,dotted, fill=black!10!white, opacity=1] (3,0) circle [radius=2];
				\draw [black!10!white, thick, dotted, fill=black!10!white, opacity=1] (-3,0) circle [radius=2];
				\draw [<->] (3,0.15) -- (5,0.15); 
				\node at (4.125,0.4) {$r$};
				\draw[-] (-6,0)--(6,0); 
				\draw[dotted,->] (0,-1.25)--(0,4.5);
				\draw[dotted] (0,-1.75)--(0,-4.5);
				\node at (6.25,0) {$x'$};
				\node at (0,4.85) {$t$};
				\draw[dashed] (-6,3)--(6,3);
				\draw[<->] (0,3.2) -- (3.755,3.2);
				\draw[<->] (0,3.2) -- (-3.755,3.2);
				\node at (1.9,3.4) {$c_{2}$}; 
				\node at (-1.9,3.4) {$c_{2}$}; 
				\node at (7,3) {$t=1$};
				\node[black] at (3,0) {\textbullet};
				\node[black] at (-3,0) {\textbullet};
				\draw[dotted,red,thick] (0,0) circle [radius=3];
				\draw[dotted,thick] (5,-4)--(-5,4);
				\draw[dotted,thick] (-5,-4)--(5,4);
				\draw[ultra thick,red!70!black] (3.8,3)--(6,3);
				\draw[ultra thick,red!70!black] (-3.85,3)--(-6,3);
				\draw[thick, red!70!black] (-3.755,3) circle [radius=0.075];
				\draw[thick, red!70!black] (3.755,3) circle [radius=0.075];
				\node[red] at (2.5,2.45) {$\mathbb{S}^{n-1}$};
				\draw[-,black] (-3,0) -- (3,0);
				\node[black!40!white] at (0,-1.5) {$\mathcal{N}$};
				\draw[black] (-0.25,-1.5) -- (-2,-1);
				\draw[black] (0.25,-1.5) -- (2,-1);
				\node[red!70!black] at (5,3.35) {$\mathcal{M}$};
				\node[red!70!black] at (-5,3.35) {$\mathcal{M}$};
				\node[red!70!black] at (-3.35,1.35) {$\pi[\mathcal{M}]$};
				\node[red!70!black] at (3.35,1.35) {$\pi[\mathcal{M}]$};
				\draw [red!70!black,thick,domain=0:39] plot ({3*cos(\x)}, {3*sin(\x)});
				\draw [red!70!black,thick,domain=141:180] plot ({3*cos(\x)}, {3*sin(\x)});
			\end{tikzpicture}
			\caption{The geometric argument underlying \eqref{eq:maingeometric} in the proof of Theorem~\ref{thm:main}. When $|y'|>c_{2}$, so $(y',1)\in\mathcal{M}:=\{(x',1)\colon\;|x'|>c_{2}\}$,  then $\pi(y',1)$, the projection of $(y',1)$ onto $\mathbb{S}^{n-1}$, belongs to $\mathcal{N}$. }
			\label{fig:geometric}
		\end{figure}
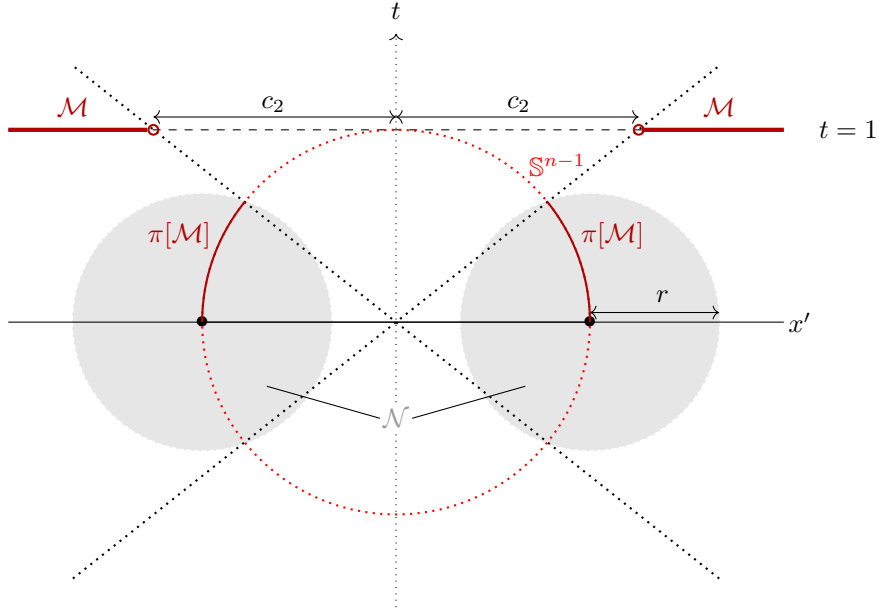
		The geometric argument depicted in Figure~\ref{fig:geometric} shows that there exists a constant $c_{2}>0$ such that 
		\begin{align}\label{eq:maingeometric}
			|y'|>c_{2}\Rightarrow\frac{(y^\prime,1)}{|(y^\prime,1)|}\in\mathcal{N}.
		\end{align}
		From the same picture we see that the orthogonal projections of such points $\tfrac{(y',1)}{|(y',1)|}$ onto $\R^{n-1}$  are contained in $\mathcal{N}$, so
		\begin{align}\label{eq:maingeometric1}
			|y'|>c_{2}\Rightarrow\frac{(y^\prime,0)}{|(y^\prime,1)|}\in\mathcal{N}.
		\end{align}
		Therefore, if $|y'|>c_{2}$, then \eqref{eq:maingeometric},  \eqref{eq:maingeometric1}, and the $(1-n)$-homogeneity of $K$ allow us to conclude for points $(y^\prime,1)$ that
		\begin{align*}
			|K(y^\prime,1)|&=|(y^\prime,1)|^{1-n}\left|K\left(\frac{(y^\prime,1)}{|(y^\prime,1)|}\right)\right|\\
			&=|(y^\prime,1)|^{1-n}\left|K\left(\frac{(y^\prime,1)}{|(y^\prime,1)|}\right)-K\left(\frac{(y^\prime,0)}{|(y^\prime,1)|}\right)\right|\\
			&\leq c_1|(y^\prime,1)|^{1-n}\left|\frac{(y^\prime,1)}{|(y^\prime,1)|}-\frac{(y^\prime,0)}{|(y^\prime,1)|}\right|^\alpha=c_1|(y^\prime,1)|^{1-n-\alpha},
		\end{align*}
		where in the second equality we used the fact that $K=0$ on $\R^{n-1}\cap\mathcal{N}$. Then, writing $c_3\coloneqq\max\{|K(y',1)|\colon\;|y'|\leq c_{2}\}$, we have that
		\begin{align*}
			\int_{\R^{n-1}}|K(y^\prime,1)|\dif y^\prime	&	\leq \int_{\{|y^\prime|\leq c_2\}} |K(y^\prime,1)|\dif y^\prime +c_1\int_{\{|y^\prime|> c_2\}}\dfrac{\dif y^\prime}{|y^\prime|^{n-1+\alpha}}\\
			&= \omega_{n-1}c_{2}^{n-1}c_{3} + c_1\int_{c_2}^\infty\dfrac{r^{n-2}\dif r}{r^{n-1+\alpha}}=\omega_{n-1}c_{2}^{n-1}c_{3}+\frac{c_1}{\alpha c_2^\alpha},
		\end{align*}
		where $\omega_{n-1}$ denotes the $(n-1)$-dimensional Lebesgue  measure of the $(n-1)$-dimensional unit ball. We conclude from \eqref{eq:L1forK} that 
		$$
		\|\mathcal T f\|_{\lebe^1(\R^{n-1})}\leq \|K(\,\cdot\,,1)\|_{\lebe^1(\R^{n-1})}\|f\|_{\lebe^1(\R^n_+)},
		$$
		so the estimate in \ref{itm:s=1} follows.
		
		\medbreak
		
		The proof of \ref{itm:s>1}, i.e., for $s>1$, follows the same structure but is more subtle. We first show that it suffices to prove that $K(\,\cdot\,,1)\in{\dot{\besov}}{^{s-1}_{1,1}}$.
		
		Let $k\coloneqq\lfloor s\rfloor +1$ (or any integer larger than $s$) so
		$$
		\Delta^k_h \mathcal T f(x',0)=\int_{\R^n_+}\Delta^k_h K(x'+y',t)f(y',t)\dif y'\dif t,
		$$
		where the finite difference acts only  on the $x'$-component. We estimate
		$$
		|\Delta^k_h  \mathcal T f(x',0)(x',0)|\leq\int_{\R^n_+}|\Delta^k_h K(x'+y',t)||f(y',t)|\dif y'\dif t.
		$$
		By Fubini's theorem, we have
		\begin{align*}
			\|\mathcal T f(\,\cdot\,,0)\|_{\dot\besov{^{s-1}_{1,1}}(\R^{n-1})}&\leq \int_{\R^n_+}\int_{\R^{n-1}}\int_{\R^{n-1}}|\Delta^k_h K(x'+y',t)|\dfrac{\dif x'\dif h}{|h|^{n+s-2}}|f(y',t)|\dif y'\dif t\\
			&= \int_{\R^n_+}\int_{\R^{n-1}}\int_{\R^{n-1}}|\Delta^k_h K(x',1)|\dfrac{\dif x'\dif h}{|h|^{n+s-2}}|f(y',t)|\dif y'\dif t
			\\
			&=\|K(\,\cdot\,,1)\|_{\dot\besov{^{s-1}_{1,1}}(\R^{n-1})}\|f\|_{\lebe^1(\R^n_+)}.
		\end{align*}
		where the first equality follows by a simple change of variable and the homogeneity of the kernel $K$. It remains to show that $K(\,\cdot\,,1)\in{\dot{\besov}}{^{s-1}_{1,1}}$, i.e., that
		$$
		\int_{\R^{n-1}}\int_{\R^{n-1}}|\Delta^k_h K(x',1)|\dfrac{\dif x'\dif h}{|h|^{n+s-2}}<\infty.
		$$
		To simplify the proof, we will endow $\R^n$ with the $\ell^1$-norm.
		Since \(K \in \hold^{k} (\R^n \setminus \{0\})\),
		we have for every \(x=(x',t) \in \R^n\) with \(\vert x\vert = 1\) and \(h \in \R^n\) such that \(\vert h \vert \leq 1/(2k)\),
		\begin{equation}\label{eq_the4Aeg4johch9Oojae1caiL}
			\vert \Delta^k_h K (x',t)\vert
			\leq c_1 \vert h \vert^k.
		\end{equation}
		By homogeneity of \(K\), for every \(x',\, h \in \R^{n - 1}\) with \(\vert h \vert \leq  (1+\vert x' \vert)/(2k)\), letting \(t = 1/({1 + \vert x'\vert)}\), we have \(\vert(t x', t)\vert = 1\) and thus by \eqref{eq_the4Aeg4johch9Oojae1caiL}
		\begin{equation*}
			\vert \Delta^k_h K (x', 1)\vert
			= t^{n - s} \vert \Delta^k_{th} K (tx', t)\vert
			\leq c_1 t^{n - s} \vert t h\vert^k =  \frac{c_1\vert h \vert^k}{(1 + \vert x'\vert)^{n+k-s}}.
		\end{equation*}
		Hence, we have 
		\begin{align}
			\label{eq_go3quo6Bu2uj1sa5Ookeiche}
			\begin{split}
				\int_{\vert h \vert \leq 1/(2k)} &\int_{\R^{n-1}}
				\frac{|\Delta^k_h K(x',1)|}{|h|^{n+s-2}} \dif x'\dif h\\
				&\leq \int_{\vert h \vert \leq 1/(2k)} \int_{\R^{n-1}}
				\frac{c_1}{|h|^{n+s-2-k}(1 + \vert x'\vert)^{n+k-s}} \dif x'\dif h <\infty.
			\end{split}
		\end{align}
		
		Next, since \(K = 0\) on \(\R^n_-\), we have for every \(x=(x',t)  \in \R^n\) such that \(\vert x \vert = 1\), 
		\begin{equation}
			\label{eq_mohrae7ood2Tad3zaeloo8Ai}
			\vert K (x',t)\vert \leq c_2 |t|^k.
		\end{equation}
		Letting \(t = 1/({1 + \vert x\vert})\), we have \(\vert(t x', t)\vert = 1\) and thus by \eqref{eq_mohrae7ood2Tad3zaeloo8Ai}
		\begin{equation*}
			\vert K (x', 1)\vert = t^{n - s} \vert K \bigl(tx', t \bigr)\vert
			\leq c_2 t^{n+k-s} = \frac{c_2}{(1 + \vert x'\vert)^{n+k-s}},
		\end{equation*}
		and thus, since \(k \ge 2\),  
		\begin{align}
			\begin{split}
				\label{eq_diveshiabeakaiqu2nooShie}
				\int_{\vert h \vert \ge 1/(2k)} &\int_{\R^{n-1}}
				\frac{|\Delta^k_h K(x',1)|}{|h|^{n+s-2}} \dif x'\dif h\\
				& \leq \sum_{i = 0}^k \binom{k}{i} \int_{\vert h \vert \ge 1/(2k)} \int_{\R^{n-1}}
				\frac{|K(x' + ih,1)|}{|h|^{n+s-2}}\dif x'\dif h \\
				& \leq 2^k \int_{\vert h \vert \ge 1/(2k)} \int_{\R^{n-1}}
				\frac{c_2}{|h|^{n + s - 2}(1 + \vert x'\vert)^{n+k-s}} \dif x'\dif h <\infty.
			\end{split}
		\end{align}
		The conclusion follows by adding together inequalities \eqref{eq_go3quo6Bu2uj1sa5Ookeiche} and \eqref{eq_diveshiabeakaiqu2nooShie}.
	\end{proof}
	\begin{remark}
		Our approach to the trace inequalities given in Proposition \ref{prop:trace} can be used to give a very simple argument to the classical Theorem \ref{thm:gag+usp}. On the one hand, the variant of the Sobolev integral formula in Proposition \ref{prop:sobolev} gives a representation formula in ${\dot{\sobo}}{^{k,1}}$ supported on a pointed cone; on the other hand, the elementary argument in the proof of Proposition \ref{prop:trace} reduces the traces inequalities to checking that $\hold_c^\infty$ is contained in ${\lebe}{^1}$ and ${\dot{\besov}}{^{k-1}_{1,1}}$.
	\end{remark}

	\section{Proofs of the trace theorems}\label{sec:trace}
	In this section, we prove the trace Theorems~\ref{thm:trace_k=1} and \ref{thm:trace_k>1}, which we present here in unified form. Recall the space 
	$$
	\mathrm{T}_k(H_\nu ,V) \coloneqq \left\{(f_0,f_1,\ldots,f_{k-1})\colon \begin{array}{c}f_{k-1}\in\lebe^1(H_\nu ,V),\\ f_j\in\dot{\besov}{^{k-1-j}_{1,1}}(H_\nu ,V),\,0\leq j\leq k-2 \end{array} \right\},
	$$
	which is Banach with respect to the canonical norm
	\begin{align*}
		\|(f_0,f_1,\ldots,f_{k-1}) \|_{\mathrm{T}_{k}(H_{\nu})}:= \Big(\sum_{j=0}^{k-2} \|f_{j}\|_{{\dot{\besov}}{_{1,1}^{k-1-j}}(H_\nu)} \Big) + \|f_{k-1}\|_{\lebe^{1}(H_{\nu})}.
	\end{align*}
	In the sequel, we denote for $u\in\hold_{c}^{\infty}(\R^{n},V)$
	$$
	\tr_k u\coloneqq(u,\partial_\nu u,\ldots,\partial_\nu^{k-1}u)\big|_{H_\nu}.
	$$
	We will now prove the following:
	\begin{theorem}\label{thm:trace_unified}
		Let \(n \ge 2\) and $\A$ be a differential operator as in \eqref{eq:A}  of order $k\geq 1$. Then $\A$ is boundary elliptic in direction $\nu$ if and only if there exists a constant $c>0$ such that the estimate
		$$
		\|\tr_k u\|_{\mathrm T_k(H_\nu)}\leq c\|\A u\|_{\lebe^1(H_\nu^+)}
		$$
		holds $\text{for all }u\in\hold^\infty_c(\R^n,V)$.
	\end{theorem}
	\begin{remark}\label{rmk:traces}
		We remark that since $u$ admits traces in $\dot{\besov}{^{k-1}_{1,1}}(H_\nu)$, all the derivatives of $u$ in the tangential direction have the suitable Besov regularity. Writing $D_\tau$ for the gradient in the tangential direction (of $H_\nu$), we have that $D_\tau^j u$ admits traces in $\dot{\besov}{^{k-1-j}_{1,1}}(H_\nu)$ for \textbf{all} $j=0,1,\ldots, k-1$. We can thus write down as a corollary of Theorem \ref{thm:trace_unified} the following estimates for $u\in\hold^\infty_c(\R^n,V)$:
		\begin{align*}
			\|D^j u\|_{{\dot{\besov}}{^{k-j-1}_{1,1}}(H_{\nu})}&\leq c\|\A u\|_{\lebe^1(H_{\nu}^{+})}\quad\text{for }j=0,1,\ldots, k-2,\\
			\|D_\tau^{k-1} u\|_{{\dot{\besov}}{^{0}_{1,1}}(H_{\nu})}&\leq c\|\A u\|_{\lebe^1(H_{\nu}^{+})},\\
			\|\partial_n^{k-1} u\|_{\lebe^1(H_{\nu})}&\leq c\|\A u\|_{\lebe^1(H_{\nu}^{+})}.
		\end{align*}
		In particular, the \textbf{only}  derivative of order at most $(k-1)$ for which the trace lacks Besov regularity is the pure $(k-1)$th normal derivative.
	\end{remark}

	For the remainder of the paper, we suppress the subscript from the notation for $H_\nu$, $H^\pm_\nu$ and write $x=(x^\prime,t)$ for a representation of $x\in\R^n$ in $H,\,H^\perp$ coordinates.
	We begin by proving necessity of boundary ellipticity. 
	\begin{lemma}\label{lem:nec_bdry_ell}
		Let \(n \ge 2\) and $\A$ be a differential operator as in \eqref{eq:A}   of order $k\geq 1$. Suppose that
		there exists a constant $c>0$ such that the estimate
		\begin{align}\label{eq:bogoargentino}
			\|u\|_{\dot{\sobo}{^{k-1,1}}(H_\nu)}\leq c\|\A u\|_{\lebe^1(H_\nu^+)}
		\end{align}
		holds for all $u\in\hold^\infty_c(\R^n,V)$. Then $\A$ is boundary elliptic in direction $\nu$.
	\end{lemma}
	\begin{proof}
		The proof relies on the fact that, if $\A$ is elliptic but not boundary elliptic in direction $\nu$, there exists $\eta\in\R^{n}$ such that $\A(\eta+\imag\nu)v=0$ for some $v\in V+\imag V$. 
		We consider separately the cases when $\eta$ and $\nu$ are linearly independent or not.
		
		\medbreak
		
		If $\eta$ and $\nu$ are linearly independent, we will use coordinates $x=(x_1,x_2,x^{\prime\prime})\in\R^n$, where $x_1=x\cdot\nu$, $x_2=x\cdot\eta$, and $x^{\prime\prime}\in\{\nu,\eta\}^\perp$.
		In this notation, we have that maps $u(x)=f(x_1+\imag x_2)v$ satisfy $\A u(x)=0$ whenever $f$ is holomorphic at $x_1+\imag x_2$ (see, e.g.,  \cite[Lem.~3.2]{GR} or \cite[Lem.~2.5]{BDG}).
		
		We will use an idea originating in the necessity proof of \cite[Thm.~V]{Aronszajn}. We choose $f=f_\varepsilon\colon\C\setminus(-\infty,-2\varepsilon]\rightarrow\C$ be a primitive of $(z+2\varepsilon)^{-1}$ for some $\varepsilon\in(0,1)$. We mean this in the following sense: let $f^{(k-1)}_\varepsilon(z)=(z+2\varepsilon)^{-1}$, where the exponent denotes $k-1$ complex derivatives. For $k>1$, this procedure requires choosing a branch of the logarithm, hence the restriction on the domain of $f_\varepsilon$.
		
		We write $u_\varepsilon(x)=f_\varepsilon(x_1+\imag x_2)v$. We consider cubes $Q_\varepsilon=(-\varepsilon,1)\times(-1,1)^{n-1}$ and a cut-off function $\rho\in\hold^\infty_c((-2,2)^n)$ such that $\rho=1$ in $[-1,1]^n$. We also choose $\varphi_\varepsilon\in\hold^\infty(\R^n)$ be such that $\spt\varphi_\varepsilon\subset\{x\in\R^n\colon x_1>-2\varepsilon\}$ and $\varphi_\varepsilon=1$ in $\{x\in\R^n\colon x_1\geq-\varepsilon\}$. Finally, we set $\psi_\varepsilon=\rho\varphi_\varepsilon$, which has the crucial properties that:
		\begin{enumerate}
			\item[(i)] $\psi_\varepsilon=1$ in $Q_\varepsilon$;
			\item[(ii)] for $|\alpha|\leq k$, we have $|\partial^\alpha\psi_\varepsilon(x)|= \|\partial^\alpha \rho\|_{\lebe^\infty}=c$ for $x\in H^+$;
			\item[(iii)] $\psi_\varepsilon u_\varepsilon\in\hold^\infty_c(\R^n,V)$.
		\end{enumerate}
		The latter implies that $\psi_\varepsilon u_\varepsilon$ is admissible for the estimate in \eqref{eq:bogoargentino}. We compute:
		$$
		\A(\psi_\varepsilon u_\varepsilon)=\psi_\varepsilon \A u_\varepsilon+\sum_{\substack{|\alpha| + |\beta|=k\\ \vert \alpha \vert < k}}c_{\alpha,\beta}\partial^\alpha u_\varepsilon\partial^\beta \psi_\varepsilon.
		$$
		Since $\A u_{\varepsilon}=0$, this implies that
		\begin{align*}
			\|\A(\psi_\varepsilon u_\varepsilon)\|_{\lebe^1(H^+)}&\leq c\sum_{|\alpha|\leq k-1}\|\partial^\alpha u_\varepsilon\|_{\lebe^1(H^+\cap(-2,2))^n}.
		\end{align*}
		Due to the structure of $f_\varepsilon$, the most singular term on the right hand side is no worse than
		$$
		\int_{(-2,2)^2}\dfrac{\dif x}{|(x_1,x_2)|},
		$$
		which is clearly finite. On the other hand, we have that
		\begin{align*}
			\int_H |D^{k-1}u_\varepsilon(0,x_2,x^{\prime\prime})|\dif\,(x_2,x^{\prime\prime})&\geq \int_{[-1,1]^{n-2}}\int_{-1}^1\dfrac{\dif x_2}{|(2\varepsilon,x_2)|}\dif x^{\prime\prime}\sim\mathrm{arsinh}(\tfrac{1}{\varepsilon}) \to \infty
		\end{align*}
		as $\varepsilon\searrow0$. Thus, we have obtained a contradiction when $\eta$ is not parallel to $\nu$.
		
		\medbreak
		
		If $\eta$ and $\nu$ are linearly dependent, then \(\A(\nu)v = 0\), and we proceed similarly to the previous case, defining now \(u_\varepsilon (x) = g(x \cdot \nu) v\), with a function \(g \in \hold^\infty (\R)\) chosen in such a way that \(g (0) = 1\).
	\end{proof}
	
	The proof of the trace theorem is  easily ensembled from the blocks we have:
	\begin{proof}[Proof of Theorem~\ref{thm:trace_unified}]
		The necessity of boundary ellipticity follows at once from Lemma~\ref{lem:nec_bdry_ell}. Assume next that $\A$ is boundary elliptic in direction $\nu$. We can identify $H_\nu$ with $\R^{n-1}$ and $H_\nu^+$ with $\R^n_+$. Let $u\in\hold^\infty_c(\R^n,V)$, so by Theorem \ref{thm:smith}, we have that we can write
		$$
		u(x)=\int_{\R^n_+} K(x+y)\A u(y)\dif y,
		$$
		where $K$ is smooth away from zero, $(k-n)$-homogeneous and vanishes on $\R^n_-$. Let $j=0,1,\ldots, k-1$. Then
		$$
		\partial_n^j u(x)=\int_{\R^n_+} K_j(x+y)\A u(y)\dif y,\quad\text{where }K_j=\partial_n^j K.
		$$
		Therefore $K_j$ is smooth away from zero, $(k-j-n)$-homogeneous and vanishes on $\R^n_-$. We can thus apply Proposition \ref{prop:trace} with $s=k-j\geq 1$ and obtain the trace inequalities
		\begin{align*}
			\|\partial_n^j u\|_{{\dot{\besov}}{^{k-j-1}_{1,1}}(\R^{n-1})}&\leq c\|\A u\|_{\lebe^1(\R^n_+)}\quad\text{for }j=0,1,\ldots, k-2,\\
			\|\partial_n^{k-1} u\|_{\lebe^1(\R^{n-1})}&\leq c\|\A u\|_{\lebe^1(\R^n_+)}.
		\end{align*}
		The proof is complete.
	\end{proof}

	\section{Proof of the Sobolev estimate}\label{sec:sob}
	In this section we prove the main Theorem~\ref{thm:main}. The proof follows from the trace Theorems \ref{thm:trace_k=1} and \ref{thm:trace_k>1} and the fact that boundary ellipticity implies cancellation. This latter observation is presented in the following:
	
	\begin{proposition}\label{prop:bc_ec}
		If \(n \ge 2\), if $\A$ be an operator which is elliptic and boundary elliptic in some direction $\nu\in\mathbb{S}^{n-1}$, then $\A$ is canceling, i.e., \(\A \) satisfies \eqref{eq:canc}.
	\end{proposition}
	
	The assumption that \(n \ge 2\) is essential, as for \(n = 1\) there are no canceling operators.
	
	\begin{proof}[Proof of Proposition~\ref{prop:bc_ec}]
		Without loss of generality, we can say that $\A$ is boundary elliptic in direction $e_1$, where $\{e_j\}_{j=1}^n$ is a standard orthonormal basis of $\R^n$. Define the operator $\A_1$ by $\A_1(\xi)=\A(\xi)$ for $\xi\in\mathrm{span}\{e_1,e_2\}\simeq\R^2$. By definition, $\A_1$ is elliptic and boundary elliptic in direction $e_1$ on $\R^2$. We claim that $\A_1$ is $\C$-elliptic, i.e.,
		\begin{align}\label{eq:C-ell}
			\ker_\C \A_1(\xi_1+\imag \xi_2)=\{0\}
		\end{align}
		for all linearly independent $\xi_1,\,\xi_2\in\R^2$ (if $\xi_1,\,\xi_2$ are not linearly independent, then \eqref{eq:C-ell} follows by ellipticity and homogeneity of $\A$). By linear independence of $\xi_i$, we can find $\lambda\in\C\setminus\{0\}$ such that $\lambda(\xi_1+\imag\xi_2)=\xi+\imag e_1$, for some $\xi\in\R^2$. In fact, writing $\xi_{1}=(\xi_{11},\xi_{21})$ and $\xi_{2}=(\xi_{12},\xi_{22})$, we have $\lambda(\xi_1+\imag\xi_2)=\xi+\imag e_1$ for some $\lambda=\mathrm{Re}(\lambda)+\imag\mathrm{Im}(\lambda)\in\mathbb{C}\setminus\{0\}$ and some $\xi\in\R^{2}$ if and only if 
		\begin{align*}
			\begin{cases}
				\xi_{11}\mathrm{Im}(\lambda)+\xi_{12}\mathrm{Re}(\lambda) = 1, \\ 
				\xi_{21}\mathrm{Im}(\lambda)+\xi_{22}\mathrm{Re}(\lambda) = 0, 
			\end{cases}
		\end{align*}
		and this is clearly solvable for $(\mathrm{Re}(\lambda),\mathrm{Im}(\lambda))\in\R^{2}\setminus\{0\}$ by the linear independence of $\xi_{1},\xi_{2}$. Now, by homogeneity of $\A$, we have that $\ker_\C\A_{1}(\xi_1+\imag\xi_2)=\ker_\C\A_{1}(\xi+\imag e_1)$, so that \eqref{eq:C-ell} holds by boundary ellipticity of $\A_1$ in direction $e_1$.
		
		It then follows by \cite[Lem.~3.2]{GR} or \cite{GRVS} that $\A_1$ is canceling, so that
		\begin{align*}
			\bigcap_{\xi\in\R^n\setminus\{0\}}\mathrm{im\,}\A(\xi)
			\subset\bigcap_{\xi\in\mathrm{span}\{e_1,e_2\}\setminus\{0\}}\mathrm{im\,}\A(\xi)
			=\bigcap_{\xi\in\mathrm{span}\{e_1,e_2\}\setminus\{0\}}\mathrm{im\,}\A_1(\xi)
			=\{0\},
		\end{align*}
		which concludes the proof.
	\end{proof}
	We can now proceed with the
	\begin{proof}[Proof of Theorem~\ref{thm:main}]
		We will now show that \ref{itm:main_b} holds, using both \ref{itm:main_a} and Theorem \ref{thm:trace_unified}. Let $u\in\hold^\infty_c(\R^n,V)$. Using the extension Theorem \ref{thm:extension_many_derivatives}, we find $U\in\dot\sobo{^{k,1}}(H^-,V)$ such that $D^j U|_{H}=D^j u|_{H}$ for $j=0,1,\ldots k-1$ (cf. Remark \ref{rmk:traces}) and
		\begin{equation}
			\label{eq_aequo6taa9oohied6waotahN}
			\|D^{k}U\|_{\lebe^1(H^-)}\leq c\|\tr_k u\|_{\mathrm T_k(H)}.
		\end{equation}
		Define an extension operator $Eu$ by $u$ in $H^+$ and by $U$ in $H^-$. We now check that $\A Eu\in\lebe^1 (\R^{n},W)$. This will enable us to use the full-space estimate \cite[Thm.~1.3]{VS}. Let $\varphi\in\hold^\infty_c(\R^n,W)$. Then, letting $\A^{*}\coloneqq(-1)^k\sum_{|\alpha|=k}A_{\alpha}^{*}\partial^\alpha$ be the formal adjoint of $\A$, we conclude
		\begin{equation}
			\label{eq_em2igaikae8ceizae8Viuphe}
			\begin{split}
				\int_{\R^n} Eu\cdot \A^*\varphi\dif x=&\int_{H^+}u\cdot \A^*\varphi\dif x+\int_{H^-}U\cdot \A^*\varphi\dif x\\
				=&\int_{H^+}\A u \cdot\varphi\dif x +\int_{H^-} \A U\cdot\varphi\dif x,
			\end{split}
		\end{equation}
		where the boundary terms in the integration by parts vanish since the traces of $D^ju$ and $D^jU$ coincide for $j=0,1,\ldots, k-1$ by construction. As a consequence, $\A Eu \in\lebe^{1}(\R^{n},W)$ with 
		$$
		\A Eu=
		\begin{cases}
			\A u&\text{in }H^+\\
			\A U&\text{in }H^-.
		\end{cases}$$ 
		We then estimate from \eqref{eq_aequo6taa9oohied6waotahN} and \eqref{eq_em2igaikae8ceizae8Viuphe} 
		\begin{equation*}
			\begin{split}
				\|D^{k-1}u\|_{\lebe^{\frac{n}{n-1}}(H^+)}&\leq \|D^{k-1}Eu\|_{\lebe^{\frac{n}{n-1}}(\R^n)}\leq c\|\A Eu\|_{\lebe^1(\R^n)}\\ & =c\left(\|\A u\|_{\lebe^1(H^+)}+\|\A U\|_{\lebe^1(H^-)}\right)\\
				&\leq c\left(\|\A u\|_{\lebe^1(H^+)}+\|\tr_k u\|_{\mathrm T_k(H^-)}\right) \leq c\|\A u\|_{\lebe^1(H^+)},
			\end{split}
		\end{equation*}
		where in the second inequality we used \cite[Thm.~1.3]{VS} and the fact that boundary ellipticity implies cancellation (see Proposition~\ref{prop:bc_ec}); in the last inequality we used Theorem \ref{thm:trace_unified}. The proof of \ref{itm:main_b} is complete.
		
		To prove the converse, first note that necessity of ellipticity for the estimate follows from \cite[Cor. 5.2]{VS}. Therefore, assume that $\A$ is elliptic, but not boundary elliptic.
		We conclude the proof of the main result by showing that the estimate in \ref{itm:main_b} must fail.
		We keep most of the construction in the proof of Lemma~\ref{lem:nec_bdry_ell}, with the only modification that $f_\varepsilon\colon\C\setminus(-\infty,-2\varepsilon]\rightarrow \C$ is given by $f^{(k-1)}_\varepsilon(z)=(z+2\varepsilon)^{\alpha}$, $\alpha=-\frac{2(n-1)}{n}$, where we choose a branch of $z^{-\alpha}$ according to the domain of $f_\varepsilon$. The remaining details are left to the keen reader.
	\end{proof}
	In following the same ideas and using the full space estimates for (weakly) canceling operators in \cite{VS,BVS,R_tams}, we can prove a broader class of estimates:
	\begin{theorem}\label{thm:estimates}
		Let \(n \ge 2\) and $\A$ be a $k$th order differential operator as in \eqref{eq:A}. Suppose that $\A$ is elliptic and boundary elliptic in direction $\nu$. Let $s\in(0,n)$ be such that $s\leq k$, $j=1,2,\ldots,\min\{k,n-1\}$, and $q\in[1,n/(n-j)]$. Then the following estimates hold
		\begin{align*}
			\|u\|_{{\dot{\sobo}}{^{k-s,\frac{n}{n-s}}}(H^+_\nu)}&\leq c\|\A u\|_{\lebe^1(H^+_\nu)}\\
			\||\,\cdot\,|^{n-j-n/q}D^{k-j}u\|_{\lebe^q(H^+_\nu)}&\leq c\|\A u\|_{\lebe^1(H^+_\nu)}\\
			\|D^{k-n}u\|_{\lebe^\infty(H^+_\nu)}&\leq c\|\A u\|_{\lebe^1(H^+_\nu)}\quad\text{when } k\geq n
		\end{align*}
		for $u\in\hold_c^\infty(\R^n,V)$.
	\end{theorem}
	Here we make the convention that the fractional scale ${\dot{\sobo}}{^{s,p}}$ is completed with the classical Sobolev spaces when $s$ is a positive integer and with the Lebesgue spaces when $s=0$.
	\begin{remark}
		The $\lebe^\infty$-estimate in Theorem \ref{thm:estimates} can be proved independently of the machinery used to prove Theorem \ref{thm:main}. One can simply use the representation formula of Theorem \ref{thm:smith} to note that $D^{k-n}u$ can be represented by the convolution of a bounded kernel with $\A u$. In particular, the $\lebe^\infty$-estimate is true in the absence of ellipticity, which is necessary for the higher order estimate of Theorem \ref{thm:main}.
	\end{remark}
	\section{Boundary ellipticity, trace operators, and examples}\label{sec:BVA}
	In this section we classify the boundary ellipticity among related conditions that lead to trace or Sobolev-type inequalities on domains or the entire space, respectively, display the consequences for spaces of functions defined in terms of the differential operators $\A$ and discuss several examples. 
	
	Based on Proposition~\ref{prop:bc_ec}, we first obtain the following implications for $\nu\in\mathbb{S}^{n-1}$ and operators $\A$ of the form \eqref{eq:A}:
	\begin{align}\label{eq:implicationchain}
		{\text{$\A$ is $\mathbb{C}$-elliptic $\Longrightarrow$ $\A$ is boundary elliptic in direction $\nu$ $\Longrightarrow$  $\A$ is canceling.}}
	\end{align}
	To connect the consequences of the preceding chain of implications with previously known results, we first restate the inequalities of the preceding sections in the language of trace operators. From a function space perspective and in order to provide a unifying framework for problems arising, e.g., in elasticity or plasticity \cite{FuchsSeregin,StrangTemam}, it is convenient to put for an open set $\Omega\subset\R^{n}$
	\begin{align}\label{eq:funcspacesDef}
		\begin{split}
			\sobo^{\A,1}(\Omega)\coloneqq\{u\in\sobo^{k-1,1}(\Omega,V)\colon\;\A u\in\lebe^{1}(\Omega,W)\},\\
			\bv^{\A}(\Omega)\coloneqq\{u\in\sobo^{k-1,1}(\Omega,V)\colon\;\A u\in\mathcal{M}(\Omega,W)\}, 
		\end{split}
	\end{align}
	where $\mathcal{M}(\Omega,W)$ denotes the $W$-valued Radon measures $\mu$ with finite total variation $\|\mu\|_{\mathcal{M}(\Omega)}$ on $\Omega$. We define the corresponding norms or metrics on $\sobo^{\A,1}$ or $\bv^{\A}$, respectively, by 
	\begin{align*}
		&\|u\|_{\sobo^{\A,1}(\Omega)}:=\|u\|_{\sobo^{k-1,1}(\Omega)}+\|\A u\|_{\lebe^{1}(\Omega)}&\text{for}\;u\in\sobo^{\A,1}(\Omega),\\ 
		&\|u\|_{\bv^{\A}(\Omega)}:=\|u\|_{\sobo^{k-1,1}(\Omega)}+\|\A u\|_{\mathcal{M}(\Omega)}&\text{for}\;u\in\bv^{\A}(\Omega),\\ 
		&d_{\A}(u,v):=\|u-v\|_{\sobo^{k-1,1}(\Omega)}+|\,\|\A u\|_{\mathcal{M}(\Omega)}-\|\A v\|_{\mathcal{M}(\Omega)}|&\text{for}\;u,v\in\bv^{\A}(\Omega),
	\end{align*}
	and note that approximation by smooth functions of $u\in\bv^{\A}(\Omega)$ can only be expected with respect to $d_{\A}$ but not $\|\cdot\|_{\bv^{\A}}$ (see e.g., \cite[Sec.~2.4]{BDG} and \cite[Sec.~2.3]{RS}). Because of this circumstance, we give the detailled proof of  
	\begin{corollary}[Refined trace theorem]\label{cor:BVAtrace}
		Let $\A$ as in \eqref{eq:A} be a $k$th order operator that is boundary elliptic in direction $\nu\in\mathbb{S}^{n-1}$. Then there exists a \emph{surjective, linear trace operator} $\mathrm{tr}_{k}\colon\bv^{\A}(H_{\nu}^{+})\to \mathrm{T}_{k}(H_{\nu},V)$ which is \emph{continuous with respect to $d_{\A}$}.
		More precisely, there exists $c=c(\A,\nu)>0$ such that the estimate
		\begin{align}\label{eq:traceestBVA}
			\|\tr_k (u)\|_{\mathrm T_k(H_\nu)}\leq c\|\A u\|_{\mathcal M(H_\nu^+)}
		\end{align}
		holds for all $u\in\bv^\A(H^+_\nu)$.
	\end{corollary} 
	\begin{proof}
		Given $u\in\sobo^{\A,1}(H_{\nu}^{+})$, we may follow \cite[\S 5.3.3]{Evans} and consider for $\varepsilon>0$ the maps $u^{\varepsilon}(x):=u(x+\varepsilon\nu)$ for $x\in H_{\nu}^{+}$. Passing to the mollifications $u_{\varepsilon}:=\rho_{\varepsilon/2}*u^{\varepsilon}$ with the $\varepsilon$-rescaled variant of a standard mollifier $\rho$ then yields that $u_{\varepsilon}\in\hold^{\infty}(\overline{H_{\nu}^{+}},V)$ and $\|u_{\varepsilon}-u\|_{\sobo^{\A,1}(H_{\nu}^{+})}\to 0$ as $\varepsilon\searrow 0$. Hence $\hold^{\infty}(\overline{H_{\nu}^{+}},V)\cap\sobo^{\A,1}(H_{\nu}^{+})$ is dense in $\sobo^{\A,1}(H_{\nu}^{+})$ for the norm topology. On the other hand, whenever $\eta_{R}\in\hold_{c}^{\infty}(\R^{n};[0,1])$ satisfies $\mathbbm{1}_{\ball_{R}(0)}\leq \eta_{R}\leq\mathbbm{1}_{\ball_{2R}(0)}$ together with 
		\begin{align*}
			|\nabla^{l}\eta_{R}|\leq \frac{c}{R^{l}}\qquad\text{for all}\;l\in\{0,...,k\}, 
		\end{align*}
		then $\eta_{R}u\to u$ as $R\to\infty$ for the norm topology on $\sobo^{\A,1}(H_{\nu}^{+})$. Combining both statements yields that $\hold_{c}^{\infty}(\overline{H_{\nu}^{+}},V)$ is dense in $\sobo^{\A,1}(H_{\nu}^{+})$ for the norm topology.  
		
		For $u\in\sobo^{\A,1}(H_{\nu}^{+})$, we may thus pick a sequence $(u_{j})\subset\hold_{c}^{\infty}(\overline{H_{\nu}^{+}},V)$ such that $u_{j}\to u$ for the $\sobo^{\A,1}(H_{\nu}^{+})$-norm. Theorem~\ref{thm:trace_unified} then implies that $(\mathrm{tr}_{k}u_{j})$ is Cauchy in $\mathrm{T}_{k}(H_{\nu},V)$ and, for $\mathrm{T}_{k}(H_{\nu},V)$ is Banach, converges to some element $\mathrm{tr}_{k}u\in\mathrm{T}_{k}(H_{\nu},V)$. By a routine argument, one sees that this element $\mathrm{tr}_{k}u$ is independent of the approximating sequence and thus well-defined. This defines a linear and bounded trace operator $\mathrm{tr}_{k}\colon\sobo^{\A,1}(H_{\nu}^{+})\to\mathrm{T}_{k}(H_{\nu},V)$, and this operator satisfies~\eqref{eq:traceestBVA} in light of~Theorem~\ref{thm:trace_unified}.
		
		For $u\in\bv^{\A}(H_{\nu}^{+})$, we choose a sequence $(v_{j})\subset \hold^{\infty}(H_{\nu}^{+},V)\cap\bv^{\A}(H_{\nu}^{+})\subset\sobo^{\A,1}(H_{\nu}^{+})$ such that $d_{\A}(v_{j},u)\to 0$ as $j\to\infty$, see \cite[Sec. 2.4]{BDG} or \cite[Sec. 2.3]{RS}. Let $r>0$ and pick a cut-off function $\varphi_{r}\in\hold^{\infty}(H_{\nu}^{+};[0,1])$ with $\varphi_{r}(x)=1$ for $x\in H_{\nu}^{+}$ with $\mathrm{dist}(x,H_{\nu})<r$, $\varphi_{r}(x)=0$ for $x\in H_{\nu}^{+}$ with $x\in H_{\nu}^{+}\setminus S_{r}$, where $S_{r}:=\{x\in H_{\nu}^{+}\colon\;\mathrm{dist}(x,H_{\nu})\leq 2r\}$ and 
		\begin{align}\label{eq:gradcondecay}		
			|\nabla^{l}\varphi_{r}|\leq\frac{c}{r^{l}}\qquad\text{for}\; l\in\{0,...,k\}.
		\end{align}
		By the construction of $\mathrm{tr}_{k}\colon\sobo^{\A,1}(H_{\nu}^{+})\to\mathrm{T}_{k}(H_{\nu},V)$, we have $\mathrm{tr}_{k}(v_{j})=\mathrm{tr}_{k}(\varphi_{r}v_{j})$ for all $j\in\mathbb{N}$ and all $r>0$. Using~\eqref{eq:traceestBVA} for $\varphi_{r}(v_{i}-v_{k})\in\sobo^{\A,1}(H_{\nu}^{+})$ as established above in the second step, we obtain for all $i,j\in\mathbb{N}$ by the Leibniz rule 
		\begin{align}\label{eq:BVAlimit}
			\begin{split}
				\|\mathrm{tr}_{k}(v_{i}-v_{j})\|_{\mathrm{T}_{k}(H_{\nu})} & = \|\mathrm{tr}_{k}(\varphi_{r}(v_{i}-v_{j}))\|_{\mathrm{T}_{k}(H_{\nu})} \\ 
				& \!\!\!\!\!\!\!\!\!\!\!\!\!\!\!\!\!\!\!\!\!\!\!\!\!\!\stackrel{\eqref{eq:gradcondecay}}{\leq} c\|\A(\varphi_{r}(v_{i}-v_{j}))\|_{\lebe^{1}(H_{\nu}^{+})}\\ 
				& \!\!\!\!\!\!\!\!\!\!\!\!\!\!\!\!\!\!\!\!\!\!\!\!\leq c\Big(\sum_{m=0}^{k-1}\frac{1}{r^{k-m}}\|D^{m}(v_{i}-v_{j})\|_{\lebe^{1}(H_{\nu}^{+})} \Big) + c(|\A v_{i}|(S_{r})+|\A v_{j}|(S_{r}))
			\end{split}
		\end{align}
		First letting $i,j\to\infty$ and then sending $r\searrow 0$, we see that $(\mathrm{tr}_{k}(v_{j}))$ is Cauchy in $\mathrm{T}_{k}(H_{\nu},V)$ and thus converges to some element of $\mathrm{T}_{k}(H_{\nu},V)$. By an argument similar to~\eqref{eq:BVAlimit}, one equally finds that this element is independent of the approximating sequence $(v_{j})$, so is well-defined, and depends linearly on $u$; note that, even though $d_{\A}$ is not translation invariant, the linearity can be obtained by a similar argument as invoked in \eqref{eq:BVAlimit}. This defines the requisite trace operator $\mathrm{tr}_{k}\colon\bv^{\A}(H_{\nu}^{+})\to\mathrm{T}_{k}(H_{\nu},V)$ which, by construction is continuous for $d_{\A}$. By construction, it coincides with the trace operator $\sobo^{k,1}(H_{\nu}^{+},V)\to\mathrm{T}_{k}(H_{\nu},V)$ on $\sobo^{k,1}(H_{\nu}^{+},V)$-maps, and hence its surjectivity follows from Theorem~\ref{thm:extension_many_derivatives}. The proof is complete. 
	\end{proof}
	
	\begin{corollary}[Refined Sobolev-type inequalities]\label{cor:SobBVA}
		Let $\A$ as in \eqref{eq:A} be a $k$th order elliptic operator that is boundary elliptic in direction $\nu$.
		Then there exists $c=c(\A,\nu)>0$ and, for any $0<s<1$, a constant $c_{s}=c(\A,\nu,s)>0$ such that the Sobolev-type estimates
		\begin{align}\label{eq:soboBVA}
			\begin{split}
				\|D^{k-1}u\|_{\lebe^{\frac{n}{n-1}}(H^+_\nu)}\leq c\|\A u\|_{\mathcal M(H^+_\nu)},\\ 
				\|D^{k-1}u\|_{\dot{\sobo}{^{s,\frac{n}{n-1+s}}}(H^{+}_{\nu})}\leq c_{s}\|\A u\|_{\mathcal{M}(H^{+}_{\nu})}
			\end{split}
		\end{align}
		hold for all $u\in\bv^\A(H^+_\nu)$.
	\end{corollary} 
	\begin{proof}
		Let $u\in\bv^{\A}(H_{\nu}^{+})$ and pick a sequence $(w_{j})\subset\hold_{c}^{\infty}(\overline{H_{\nu}^{+}},V)$ that converges to $u$ with respect to $d_{\A}$; this can be achieved by taking a sequence $(v_{j})\subset\sobo^{\A,1}(H_{\nu}^{+})$ such that $d_{\A}(v_{j},u)<\frac{1}{j}$ and then choosing, for each $j\in\mathbb{N}$, some $w_{j}\in\hold_{c}^{\infty}(\overline{H_{\nu}^{+}},V)$ such that $\|w_{j}-v_{j}\|_{\sobo^{\A,1}(H_{\nu}^{+})}<\frac{1}{j}$ as in the very first part of the previous proof. Passing to a subsequence if necessary, we may achieve $D^{k-1}w_{j}\to D^{k-1}u$ $\mathscr{L}^{n}$-a.e. in $H_{\nu}^{+}$; then~\eqref{eq:soboBVA} is a direct consequence of Fatou's lemma, Theorem~\ref{thm:estimates} and $d_{\A}(w_{j},u)\to 0$ as $j\to\infty$. 
	\end{proof} 
	We now turn to some examples that demonstrate the richness and the limitations of the boundary ellipticity. 
	\begin{example}[$\mathbb{C}$-elliptic operators]
		Based on \eqref{eq:implicationchain}, all $\mathbb{C}$-elliptic operators are boundary elliptic in any direction $\nu\in\mathbb{S}^{n-1}$. This particularly comprises the \emph{symmetric gradient}
		\begin{align}\label{eq:symgrad}
			\varepsilon(u)\coloneqq\frac{1}{2}(Du+Du^{\top}),\qquad u=(u_{1},...,u_{n})\colon\R^{n}\to\R^{n}
		\end{align}
		for $n\geq 2$ and, denoting by $\mathbbm{1}_{n\times n}$ the $(n\times n)$-unit matrix, the \emph{trace-free symmetric} or \emph{deviatoric symmetric gradient}
		\begin{align}\label{eq:tfsym}
			\varepsilon^{D}(u)\coloneqq\varepsilon(u)-\frac{\mathrm{div}(u)}{n}\mathbbm{1}_{n\times n},\qquad u=(u_{1},...,u_{n})\colon\R^{n}\to\R^{n}
		\end{align}
		in $n\geq 3$ dimensions (see, e.g., \cite[Sec.~2]{BDG}). For the symmetric gradient, Corollary~\ref{cor:BVAtrace} directly yields the halfspace version of the $\bd$-trace theorem due to Strang--Temam \cite{StrangTemam} (also see Babadjian \cite{Babadjian}); more generally, for halfspaces Corollary~\ref{cor:BVAtrace} lets us retrieve the trace theorems for $\mathbb{C}$-elliptic operators  from \cite{BDG,DieGme21,GRVS} as special cases by \eqref{eq:implicationchain}. 
	\end{example}
	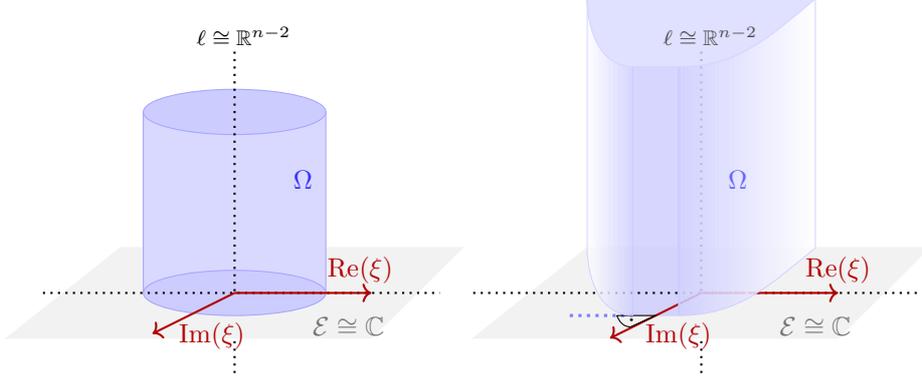
\begin{figure}
		\begin{tikzpicture}[scale=0.6]
			\draw[-,left color=white,right color=blue!50, shade, white] (-2,0) to [bend left=20] (0,-0.5) to [bend left=20] (2,0) to (-2,0);
			\draw[color=black!5!white, fill=black!5!white] (-2.5,1) -- (-5,-1)--(3,-1)--(5,1)--(-2.5,1);
			\draw[blue!35!white, fill=blue!15!white] (-2,0) -- (2,0) -- (2,4) -- (-2,4) -- (-2,0);
			\draw[blue!35!white, fill=blue!20!white] (0,0) ellipse (2cm and 0.5cm);
			\draw[blue!35!white, fill=blue!20!white] (0,4) ellipse (2cm and 0.5cm);
			\draw[thick,dotted] (0,5.33333) -- (0,0);
			\draw[thick,dotted] (-0,-1.0666)--(0,-1.77777);
			\draw[thick,dotted] (-4.2,0) -- (4.5,0);
			\draw[color=red!70!black,thick,->] (0,0) -- (3,0);
			\draw[color=red!70!black,thick,->] (0,0) -- (-1.8,-0.9);
			\node[color=red!70!black] at (2.75,0.5) {$\Re(\xi)$};
			\node[color=red!70!black] at (-0.5,-0.95) {$\Im(\xi)$};
			\node[color=black!55!white] at (2.5,-0.75) {$\mathcal{E}\cong\mathbb{C}$};
			\node[color=white!20!blue] at (1.5,2.5) {$\Omega$};
			\node at (0.2,5.666) {\footnotesize{$\ell\cong\R^{n-2}$}};
		\end{tikzpicture}
		\begin{tikzpicture}[scale=0.6]
			\draw[color=black!5!white, fill=black!5!white] (-2.5,1) -- (-5,-1)--(3,-1)--(5,1)--(-2.5,1);
			\draw[blue!50!white,dotted,very thick] (-1.5,-0.5)--(-3,-0.5);
			\draw[blue!20!white, fill=blue!20!white,opacity=0.8] (-1.5,-0.5)--(-0.5,-0.5) -- (-0.5,5) -- (-1.5,5) -- (-1.5,-0.5);
			\draw[thick,dotted] (0,5.33333) -- (0,0);
			\draw[thick,dotted] (-0,-1.0666)--(0,-1.77777);
			\draw[thick,dotted] (-5,0) -- (5,0);
			\draw[color=red!70!black,thick,->] (0,0) -- (3,0);
			\draw[color=red!70!black,thick,->] (0,0) -- (-2,-1);
			\node[color=red!70!black] at (3,0.5) {$\Re(\xi)$};
			\node[color=red!70!black] at (-0.5,-0.95) {$\Im(\xi)$};
			\node[color=black!55!white] at (2.5,-0.75) {$\mathcal{E}\cong\mathbb{C}$};
			\draw[black] (-1.85,-0.5) to [out = -70, in = 180] (-1.55,-0.75) -- (-1,-0.5) -- (-1.85,-0.5);
			\node at (-1.525,-0.615) {$\bf{\cdot}$};
			\draw[left color=blue!20,right color=white, shade, blue!20, opacity=0.8] (-0.5,-0.5) to [out=0, in =220] (2.5,1) -- (2.5,6.5) to [out=220, in =0] (-0.5,5) -- (-0.5,-0.5);
			\draw[left color=white,right color=blue!20, shade, blue!20,opacity=0.8] (-1.5,-0.5) to [out=180, in =270] (-2.5,1) -- (-2.5,6.5) to [out=270, in =180] (-1.5,5) -- (-1.5,-0.5);
			\draw[blue!15!white, fill = blue!15!white,opacity=0.8] (-2.5,6.5) to [out=270, in =180] (-1.5,5) -- (-0.5,5) to [out=0, in =220] (2.5,6.5) to (-2.5,6.5) ; 
			\node[black!70!white] at (0.2,5.666) {\footnotesize{$\ell\cong\R^{n-2}$}};
			\node[color=white!40!blue] at (0.8,2.5) {$\Omega$};
		\end{tikzpicture}
		\caption{Shifting holomorphic maps along $\R^{n-2}$. To construct domains for which there is \emph{no} trace operator $\bv^{\A}(\Omega)\to\mathrm{T}_{k}(\partial\Omega,V)$ in absence of $\mathbb{C}$-ellipticity, one picks $\xi\in\mathbb{C}^{n}\setminus\{0\}$ and $v\in (V+\mathrm{i}V)\setminus\{0\}$ such that $\A(\xi)v=0$. For suitable holomorphic functions $f\colon \C\supset\mathbb{D}\to\mathbb{C}$ (e.g. with $f^{(k-1)}(z)=\frac{1}{z-1}$ and the complex disk $\mathbb{D}$), either the real or the imaginary part of $u(x):=f(x\cdot\mathrm{Re}(\xi)+\imag x\cdot\mathrm{Im}(\xi))v$ violate the trace estimate over a set $\Omega$ that up to a rotation coincides with $\{t_{1}\mathrm{Re}(\xi)+t_{2}\mathrm{Im}(\xi)+(0,z'')\colon\;t_{1}^{2}+t_{2}^{2}<1,\;z''\in\R^{n-2}\}$ (figure to the left); see \cite[Thm.~4.18]{BDG}, \cite{GRVS}. In the same way, one can come up with domains that violate the the trace estimate for non-boundary-elliptic operators by use of Lemma~\ref{lem:nec_bdry_ell} (figure to the right). Including a straight piece orthogonal to $\mathrm{Im}(\xi)$, one sees the necessity of the boundary ellipticity even more directly.}\label{fig:holshift}
	\end{figure}
	\begin{example}[The trace-free symmetric gradient in $n=2$ dimensions]
		If $n=2$, then $\mathbb{C}$-ellipticity coincides with cancellation for first order elliptic operators, see \cite{GR}. The trace-free symmetric gradient \eqref{eq:tfsym} is known to be non-canceling for $n=2$ and therefore, in light of \eqref{eq:implicationchain}, cannot be boundary elliptic in direction $\nu$ for \emph{any} $\nu\in\mathbb{S}^{1}$. We may then explicitly verify that for no halfspace $H_{\nu}^{+}\subset\R^{2}$ the operator $\varepsilon^{D}$ admits the trace or Sobolev estimates from Corollary~\ref{cor:BVAtrace} and~\ref{cor:SobBVA}: Put 
		\begin{align*}
			f(x_{1},x_{2})\coloneqq\Big(\frac{x_{1}}{x_{1}^{2}+x_{2}^{2}},-\frac{x_{2}}{x_{1}^{2}+x_{2}^{2}} \Big),\qquad (x_{1},x_{2})\in H_{\nu}^{+}.
		\end{align*}
		An explicit computation directly verifies that $\varepsilon^{D}(f)(x_{1},x_{2})=0\in\R^{2\times 2}$ for all $(x_{1},x_{2})\in H_{\nu}^{+}$ regardless of $\nu=(\nu_{1},\nu_{2})\in\mathbb{S}^{1}$. Whenever $\eta_{R}\in\hold_{c}^{\infty}(\R^{2};[0,1])$ is a cut-off function with $\mathbbm{1}_{\ball_{R}(0)}\leq \eta_{R}\leq \mathbbm{1}_{\ball_{2R}(0)}$ and $|\nabla\eta_{R}|\leq \frac{2}{R}$ for $R>0$, then $u_{R}:=\eta_{R}f\in\bv^{\A}(H_{\nu}^{+})$. By construction, we obtain that $\sup_{R>0}\|\varepsilon^{D}(u_{R})\|_{\mathcal{M}(H_{\nu}^{+})}<\infty$, but parametrising $H_{\nu}=\R\nu^{\bot}$ with $\nu^{\bot}=(-\nu_{2},\nu_{1})$, we then obtain 
		\begin{align*}
			\int_{H_{\nu}}|u_{R}(x_{1},x_{2})|\dif\mathscr{H}^{1}(x_{1},x_{2}) \geq \int_{-R}^{R}\frac{\dif t}{|t|} = \infty
		\end{align*}
		which is in line with Corollary~\ref{cor:BVAtrace}. Similarly, one obtains with a constant $c>0$
		\begin{align*}
			\int_{H_{\nu}^{+}}|u_{R}(x)|^{2}\dif x\geq \int_{H_{\nu}^{+}\cap\ball_{R}(0)}\frac{\dif\,(x_{1},x_{2})}{x_{1}^{2}+x_{2}^{2}} \geq c\int_{0}^{R}\frac{\dif r}{r} = \infty,  
		\end{align*}
		which is in line with Corollary~\ref{cor:SobBVA}. 
	\end{example}
	Even though our main focus of the present paper is on halfspaces, let us note that the failure of boundary ellipticity of $\A$ in a certain direction can immediately be used to construct a domain $\Omega\subset\R^{n}$ for which there is no boundary trace operator $\bv^{\A}(\Omega)\to\mathrm{T}_{k}(\partial\Omega,V)$ (see Figure~\ref{fig:holshift}) with the obvious definition of the latter space via local charts. However, as it is more restrictive for an operator to not be boundary elliptic in a certain direction than to not be $\mathbb{C}$-elliptic in general (see the next example), a modification of the argument sketched in Figure~\ref{fig:holshift} directly yields that the existence of a trace operator $\bv^{\A}(\Omega)\to\mathrm{T}_{k}(\partial\Omega,V)$ forces the outward unit normals $\nu_{\partial\Omega}$ to belong to some set $K\subset\mathbb{S}^{n-1}$ depending on $\A$. While this technical point will be pursued in future work, we conclude the present paper by giving examples of operators that, in view of \eqref{eq:implicationchain}, fail to be $\mathbb{C}$-elliptic, yet are boundary elliptic in certain directions and thus admit Sobolev estimates on the corresponding halfspaces:
	\begin{example} 
		Let $n\geq 3$, $N\geq 3$, $V=\R^{N}$, $W=\R^{((N-1)n-1)\times 2}$ and consider the differential operator $\A$ acting on $u=(u_{1},...,u_{N})\colon\R^{n}\to\R^{N}$
		\begin{align}\label{eq:directional}
			\renewcommand\arraystretch{1.3}
			\A u = \left[
			\begin{array}{c|c}
				\partial_{1}u_{1}-\partial_{2}u_{2}  & \partial_{1}u_{2}+\partial_{2}u_{1} \\
				\partial_{3} u_{1} & \partial_{3}u_{2} \\ \vdots & \vdots \\ \partial_{n}u_{1} & \partial_{n}u_{2} \\
				\hline
				\nabla u_{3} & \mathbf{0} \\ \vdots & \vdots \\ \nabla u_{N} & \mathbf{0} 
			\end{array}
			\right], 
		\end{align}
		where $\mathbf{0}$ denotes the zero vector in $\R^{n}$. As established in \cite[Counterexample 3.4]{GR}, this operator serves as an example of an elliptic operator being canceling yet failing to be $\mathbb{C}$-elliptic. However, $\A$ is boundary elliptic in every direction $\nu\in\text{span}\{e_{3},...,e_{n}\}$. Based on this operator, boundary elliptic, non-$\mathbb{C}$-elliptic operators of arbitrary order can be constructed: In fact, if $\mathbb{B}$ is a $(k-1)$th order, $\mathbb{C}$-elliptic differential operator on $\R^{n}$ from $W=\R^{((N-1)n-1)\times 2}$ to some finite dimensional real vector space, then $\mathbb{B}\A$ is of $k$th order, boundary elliptic in all directions $\nu\in\mathrm{span}\{e_{3},...,e_{n}\}$ but non-$\mathbb{C}$-elliptic. 
	\end{example}

\end{document}